\documentclass[11pt,a4paper]{article}
\usepackage[T1]{fontenc}
\usepackage[utf8]{inputenc}
\usepackage{lmodern,microtype}
\usepackage[margin=27mm,headheight=14pt]{geometry}
\usepackage{amsmath,amssymb,amsthm,mathtools,bm,mathrsfs}
\usepackage{esint,aliascnt}
\usepackage{cite} 
\usepackage[colorlinks=true,linkcolor=blue,citecolor=blue,urlcolor=black]{hyperref}
\usepackage[noabbrev]{cleveref}
\usepackage{fancyhdr}
\fancypagestyle{plain}{
  \fancyhf{}
  \fancyfoot[C]{\thepage}

}
\numberwithin{equation}{section}
\newtheorem{theorem}{Theorem}[section]
\newaliascnt{proposition}{theorem}
\newtheorem{proposition}[proposition]{Proposition}
\aliascntresetthe{proposition}
\newaliascnt{lemma}{theorem}
\newtheorem{lemma}[lemma]{Lemma}
\aliascntresetthe{lemma}
\newaliascnt{corollary}{theorem}
\newtheorem{corollary}[corollary]{Corollary}
\aliascntresetthe{corollary}
\theoremstyle{definition}
\newaliascnt{definition}{theorem}
\newtheorem{definition}[definition]{Definition}
\aliascntresetthe{definition}
\newaliascnt{remark}{theorem}
\newtheorem{remark}[remark]{Remark}
\aliascntresetthe{remark}
\crefname{theorem}{Theorem}{Theorems}
\crefname{proposition}{Proposition}{Propositions}
\crefname{lemma}{Lemma}{Lemmas}
\crefname{corollary}{Corollary}{Corollaries}
\crefname{definition}{Definition}{Definitions}
\crefname{remark}{Remark}{Remarks}
\crefname{section}{Section}{Sections}
\crefname{appendix}{Appendix}{Appendices}
\AtBeginDocument{}
\newcommand{\eps}{\varepsilon}
\newcommand{\R}{\mathbb R}
\newcommand{\T}{\mathbb T^3}
\newcommand{\Oe}{{\Omega_\eps}}
\newcommand{\He}{{\mathcal H_\eps}}
\newcommand{\Ie}{{\mathcal I_\eps}}
\newcommand{\Ceps}{\mathcal C_\eps}
\newcommand{\Eeps}{\mathscr E_\eps}
\newcommand{\M}{\mathbf M}
\newcommand{\B}{\mathbf B}
\newcommand{\Div}{\operatorname{div}}
\newcommand{\Id}{\operatorname{Id}}
\newcommand{\dx}{\,\mathrm dx}
\newcommand{\dt}{\,\mathrm dt}
\newcommand{\ds}{\,\mathrm dS}
\newcommand{\dd}{\,\mathrm d}
\newcommand{\weak}{\rightharpoonup}
\newcommand{\wstar}{\stackrel{*}{\rightharpoonup}}
\newcommand{\1}{\mathbf 1}
\newcommand{\calF}{\mathcal F}
\newcommand{\calE}{\mathcal E}
\newcommand{\calA}{\mathcal A}
\newcommand{\Dt}{{\mathcal D_T}}
\newcommand{\W}{\mathcal W}
\newcommand{\Ssig}{\mathcal S_\sigma}
\newcommand{\vu}{\bm{u}}
\newcommand{\vv}{\bm{v}}
\newcommand{\vw}{\bm{w}}
\newcommand{\vz}{\bm{z}}
\newcommand{\vomega}{\bm{\omega}}
\newcommand{\Ithree}{\mathsf I}
\newcommand{\vf}{\bm{f}}
\newcommand{\vg}{\bm{g}}
\newcommand{\vpsi}{\bm{\psi}}
\newcommand{\vxi}{\bm{\xi}}
\newcommand{\ve}{\bm{e}}
\newcommand{\vn}{\bm{n}}
\newcommand{\vh}{\bm{h}}
\newcommand{\vH}{\bm{H}}
\newcommand{\vG}{\bm{G}}
\newcommand{\vZ}{\bm{Z}}
\newcommand{\vr}{\bm{r}}
\newcommand{\cV}{\mathcal V}
\newcommand{\cH}{\mathcal H}
\allowdisplaybreaks[2]
\hypersetup{pdftitle={Critical homogenization of the Navier--Stokes--Cahn--Hilliard system in perforated domains},pdfsubject={Critical holes, phase-dependent viscosity, and the Brinkman law},pdfauthor={Jiaojiao Pan and Luqi Wang}}
\title{\Large\bfseries Critical homogenization of the Navier--Stokes--Cahn--Hilliard system in perforated domains}
\author{Jiaojiao Pan\footnote{School of Mathematics, Nanjing University, Nanjing 210093, China, panjiaojiao.math@gmail.com} \and Luqi Wang\footnote{School of Mathematics, Nanjing University, Nanjing 210093, China, wangluqi@nju.edu.cn}}
\date{}
\begin{document}
\maketitle
\vspace{-1.5em}
\begin{abstract}
We consider the three-dimensional incompressible Navier--Stokes--Cahn--Hilliard (NSCH) system in domains perforated by no-slip obstacles of diameter of order \(\varepsilon^3\) separated by distances of order \(\varepsilon\). Both viscosity and mobility may depend on the phase variable. At this critical Stokes-capacity scale, the limiting velocity \(\boldsymbol u\) and phase field \(\phi\) satisfy an NSCH system with the additional Brinkman resistance \(\nu(\phi)\mathbf B\boldsymbol u\), where \(\mathbf B\) is determined by the exterior Stokes capacity of the reference obstacle. The principal analytical difficulty is the interaction between the phase-dependent viscosity and the order-one energy concentration of the critical Stokes correctors. The variational chemical-potential identity yields, after scalar extension, strong convergence of the phase in \(L^2(0,T;H^1(\Omega))\), and hence strong convergence of the phase-dependent viscosity in the same topology. Combined with a cellwise Hardy multiplier estimate, this compactness can be transferred through the concentrated corrector layer. The resulting weighted-capacity statement applies to uniformly positive and bounded coefficient sequences converging strongly in \(L^2(0,T;H^1(\Omega))\) and simultaneously identifies the effective Brinkman force and the corresponding viscous dissipation lower bound. We also investigate a vanishing capillary coefficient \(\lambda_\varepsilon\to0\). After the natural velocity scaling, the limit is an unsteady Stokes--Brinkman equation coupled to an unadvected Cahn--Hilliard equation.
\end{abstract}
\noindent\textbf{Keywords.} Homogenization; Navier--Stokes--Cahn--Hilliard system; Critical perforations; Brinkman law; Stokes capacity.

\smallskip
\noindent\textbf{2020 Mathematics Subject Classification.} 35B27, 76M50, 35Q30, 76D05.

\section{Introduction}\label{sec:introduction}
In this paper, we study a diffuse-interface model for two incompressible fluids of matched density in a dilute perforated domain. The ambient set \(\Omega\) is either a bounded \(C^3\) domain in \(\mathbb R^3\) or the unit torus \(\mathbb T^3\). The fluid region \(\Omega_\varepsilon\) is obtained by removing periodically distributed no-slip obstacles whose diameter is of order \(\varepsilon^3\), while the inter-obstacle distance is of order \(\varepsilon\). Let \(T>0\). The unknowns are the velocity \(\boldsymbol u_\varepsilon\), pressure \(p_\varepsilon\), phase field \(\phi_\varepsilon\) and chemical potential \(\mu_\varepsilon\). The functions \(\nu\), \(m\) and \(F\) denote the viscosity, mobility and bulk free-energy density, respectively. The parameter \(\lambda_\varepsilon>0\) is the capillary coefficient and \(\boldsymbol f_\varepsilon\) is the external force. With \(D\boldsymbol v=(\nabla\boldsymbol v+\nabla\boldsymbol v^{\textup T})/2\), the microscopic system is
\begin{equation}\label{eq:micro}
 \left\{\begin{aligned}
 \partial_t \vu_\eps+\Div(\vu_\eps\otimes \vu_\eps)
 -\Div(2\nu(\phi_\eps)D\vu_\eps)+\nabla p_\eps
     &=\vf_\eps-\lambda_\eps\phi_\eps\nabla\mu_\eps,\\
 \Div \vu_\eps&=0,\\
 \partial_t\phi_\eps+\Div(\phi_\eps \vu_\eps)
     &=\Div(m(\phi_\eps)\nabla\mu_\eps),\\
 \mu_\eps&=-\Delta\phi_\eps+F'(\phi_\eps).
 \end{aligned}\right.
\end{equation}
The velocity \(\boldsymbol u_\varepsilon\) satisfies the no-slip condition on the solid boundaries, while \(\phi_\varepsilon\) and \(\mu_\varepsilon\) satisfy homogeneous Neumann conditions. The precise geometry and constitutive hypotheses are given in \cref{sec:formulation}; the matched-density diffuse-interface model and its energy-level weak formulation are discussed in \cite{Abels,GG}.

Writing the obstacle diameter as \(a_\varepsilon\sim\varepsilon^\alpha\), the three-dimensional Stokes-capacity threshold is \(\alpha=3\). A single obstacle contributes resistance of order \(a_\varepsilon\), whereas an \(\varepsilon\)-periodic array contains \(O(\varepsilon^{-3})\) obstacles, so the aggregate resistance remains of order one precisely at the critical exponent. The classical perforated-domain theory identifies the corresponding Brinkman mechanism for Stokes and Navier--Stokes flow \cite{AllaireI,AllaireII}; its evolutionary critical counterpart appears in \cite{FNN}, and a unified Stokes treatment is available in \cite{Lu2020}. Related critical limits for other fluid models are developed in \cite{BO2022,Pan2025,BLMO}. For \(1<\alpha<3\), the balance enters the Darcy regime and the effective scaling is different.

Phase-field homogenization has mainly been developed in pore-scale porous media. Multiple-scale analysis with drift leads to effective Stokes--Cahn--Hilliard equations in \cite{SPPK}, while two-scale and unfolding methods yield related evolutionary effective systems in \cite{BM,LM2022}. In the dilute setting, the NSCH system with phase-dependent viscosity and mobility has been analyzed for \(\alpha>3\), where the obstacle capacity disappears \cite{LPW}; a free-slip model with an oscillating viscosity tensor and a phase source is considered in \cite{CDM}. The present regime combines the no-slip critical capacity \(\alpha=3\) with a phase-dependent viscosity. Consequently, the corrector layer carries order-one energy and the viscosity coefficient must be identified inside a concentrated Stokes layer.

A comparable variable-coefficient effect occurs for heat-conducting fluids in critically perforated domains \cite{FLS}, where temperature regularity controls the coefficient in the Brinkman term. For the NSCH system, the chemical-potential identity supplies the required coefficient compactness:
\[
 \Eeps\phi_\eps\to\phi
 \quad\text{strongly in }L^2(0,T;H^1(\Omega)),
\]
and the assumption \(\nu\in C^1(\mathbb R)\), \(\|\nu'\|_\infty<\infty\), yields
\[
 \nu(\Eeps\phi_\eps)\to\nu(\phi)
 \quad\text{strongly in }L^2(0,T;H^1(\Omega)).
\]
A cellwise Hardy estimate matches this topology to the critical corrector gradients. \Cref{prop:weighted} isolates the resulting weighted-capacity pairing for uniformly positive and bounded coefficient sequences and supplies the Brinkman contribution to the viscous lower bound. The concentration example in \cref{rem:coefficientexample} shows why space--time \(L^2\) convergence alone cannot resolve the capacity layer.

The geometric part of the proof uses localized exterior Stokes fields together with divergence repairs; quantitative dilute-Stokes estimates of a related type appear in \cite{JLP}. Oscillating-test constructions in the Darcy range, such as \cite[Proposition~2.7]{HLO}, are formulated for \(1<\alpha<3\). The critical endpoint construction needed here is developed in \cref{sec:correctors}.

The same compactness and capacity mechanism also yields a second limit when \(\lambda_\varepsilon\to0\). With \(\boldsymbol v_\varepsilon=\lambda_\varepsilon^{-1/2}\boldsymbol u_\varepsilon\), the time derivative and viscous terms remain at order one, whereas convection and capillarity vanish; the limit is an unsteady Stokes--Brinkman equation coupled to an unadvected Cahn--Hilliard equation.

 \paragraph{{\bf Key innovations and difficulties.}} The main analytical contribution is the resolution of the interaction between the phase-dependent viscosity and the critical Stokes-capacity layer at the natural energy level. The chemical-potential identity upgrades the phase convergence from strong \(C([0,T];L^2(\Omega))\) and weak \(L^2(0,T;H^1(\Omega))\) convergence to strong \(L^2(0,T;H^1(\Omega))\) convergence, while the cellwise Hardy multiplier converts this information into a coefficient-replacement estimate for the concentrated corrector gradients. Their combination yields the weighted-capacity principle in \cref{prop:weighted}. For uniformly positive and bounded coefficient sequences converging strongly in \(L^2(0,T;H^1(\Omega))\), the weighted-capacity principle identifies both the variable Brinkman resistance in the momentum equation and the corresponding lower bound for the viscous dissipation. Feireisl, Lu and Sun \cite{FLS} exploit uniform spatial H\"older control of the temperature coefficient. Here the microscopic phase need not be spatially continuous: the chemical-potential identity yields strong convergence in \(L^2(0,T;H^1(\Omega))\), and a cellwise Hardy estimate transfers this Sobolev control to the critical Stokes layer. At the endpoint \(\alpha=3\), this coefficient--capacity interaction is the feature that separates the present argument from the small-hole NSCH regime \(\alpha>3\), where the obstacle capacity vanishes \cite{LPW}. The same compactness and corrector framework also accommodates the vanishing-capillarity scaling and leads to an unsteady Stokes--Brinkman/Cahn--Hilliard limit using the same geometric correctors.

 \paragraph{{\bf Outline of the paper.}} The paper is organized as follows. \Cref{sec:formulation} introduces the perforated geometry and function spaces, constructs the scalar extension operator, formulates the microscopic and effective weak problems, and states the main homogenization results. Uniform estimates and compactness of the phase field are developed in \cref{sec:phase}, culminating in the strong \(L^2(0,T;H^1(\Omega))\) convergence needed for the phase-dependent viscosity. The critical geometric analysis is carried out in \cref{sec:correctors}, where localized Stokes correctors, the Hardy multiplier estimate, and the weighted-capacity limit are established. Using these ingredients, \cref{sec:limit} proves velocity compactness, passes to the coupled NSCH--Brinkman limit, identifies the limiting dissipation, and reconstructs the pressure. The vanishing-capillarity regime and the resulting unsteady Stokes--Brinkman/Cahn--Hilliard system are treated in \cref{sec:vanishing}. Finally, \cref{app:existence} provides the fixed-\(\varepsilon\) existence theory for the microscopic finite-energy weak solutions used throughout the homogenization argument.

\section{Preliminaries and main results}\label{sec:formulation}
\subsection{Geometry and function spaces}\label{subsec:geometry}
Fix $T>0$. We consider a bounded connected domain $\Omega \subset \R^3$ of class $C^{3}$ or the unit torus $\T=\R^3/\mathbb Z^3$. We use the standard Lebesgue and Sobolev spaces, with $H^j=W^{j,2}$. The notation $C_w([0,T];X)$ denotes weak continuity with values in $X$. For a measurable set $G$ with $0<|G|<\infty$, let $|G|$ denote its volume and $\1_G$ its characteristic function. We write
\[
 (z)_G=\fint_G z\dx=|G|^{-1}\int_Gz\dx,
 \quad L^2_0(G)=\{z\in L^2(G):(z)_G=0\}.
\]
We use $(\cdot,\cdot)_G$ for the $L^2(G)$ inner product of scalar- or vector-valued functions. When the spatial domain is clear from the context, $\|\cdot\|_p$ denotes the corresponding $L^p$-norm, with the same convention for scalar-, vector-, and matrix-valued fields. When necessary, the domain is indicated explicitly by $\|\cdot\|_{L^p(G)}$. The symbol $\langle\cdot,\cdot\rangle$ denotes duality. For $x\in\R^3$ and $r>0$, $B(x,r)$ denotes the open ball of center $x$ and radius $r$; $\#J$ is the cardinality of a finite index set $J$.

All differential operators act on the spatial variable unless a time derivative is indicated. For a vector field $\vv$ and a matrix field $A$, we write
\[
 (\nabla \vv)_{ij}=\partial_jv_i,\quad
 D\vv=\tfrac12(\nabla \vv+\nabla \vv^{\textup T}),\quad
 (\Div A)_i=\sum_{j=1}^3\partial_j A_{ij},\quad
 \Delta=\sum_{j=1}^3\partial_j^2.
\]
Moreover, we write $(\vv\otimes\vw)_{ij}=\vv_i\vw_j$ and $A:B=\sum_{i,j}A_{ij}B_{ij}$. $\Ithree$ denotes the $3\times3$ identity matrix and $\Id$ the identity on a function space.

Let $\Dt=(0,T)\times\Omega$. We abbreviate $L^p(0,T;L^q(\Omega))$ and $L^p(0,T;H^1(\Omega))$ by $L^p_tL^q_x$ and $L^p_tH^1_x$, respectively; a different spatial domain is indicated explicitly. Constants denoted by $C$ are independent of the perforation parameter, but may depend on $T$ and the fixed data. We write $C_T$ when the dependence on the final time is emphasized. A constant $C_\eps$ may depend on a fixed perforated domain, and $C_{\vpsi}$ may additionally depend on finitely many smooth norms of a test field $\vpsi$; combined subscripts, such as $C_{\vpsi,T}$, record the corresponding fixed dependencies. For $\rho>0$, the notation $O(\rho)$ denotes a quantity bounded in magnitude by $C\rho$, uniformly in the cell index.

Let $K$ be the closure of a bounded $C^3$ domain containing the origin, with connected exterior, and assume $K\subset B(0,1/8)$. Fix $\beta>0$. For a sufficiently small real parameter $\eps>0$, define $a_\eps=\beta\eps^3$ and
\[
 \ell_\eps=
 \begin{cases}
 \eps,&\Omega\subset\R^3,\\
 \lceil\eps^{-1}\rceil^{-1},&\Omega=\T.
 \end{cases}
\]
Thus $\ell_\eps/\eps\to1$ and $a_\eps/\ell_\eps^3\to\beta$. On the torus, $N_\eps=\ell_\eps^{-1}$ is an integer. Put $x_{\eps,k}=\ell_\eps k$ and
\[
 Q_{\eps,k}=x_{\eps,k}+\ell_\eps(-1/2,1/2)^3.
\]
In a bounded domain, let $\Ie=\{k\in\mathbb Z^3:\overline{Q_{\eps,k}}\subset\Omega\}$. On the torus, take $\Ie=\{0,\ldots,N_\eps-1\}^3$, with centers and cells understood modulo $\mathbb Z^3$. Define the individual obstacles, solid set, fluid domain and fluid indicator by
\begin{equation}\label{eq:geometry}
 K_{\eps,k}=x_{\eps,k}+a_\eps K,\quad
 \He=\bigcup_{k\in\Ie}K_{\eps,k},\quad
 \Oe=\Omega\setminus\He,\quad \chi_\eps=\1_{\Oe}.
\end{equation}
For $\eps$ small enough, the holes are disjoint and $\Oe$ is connected. Indeed, if $R_K=\sup_{y\in K}|y|<1/8$, then $K_{\eps,k}\subset B(x_{\eps,k},R_Ka_\eps)$ and distinct inclusions satisfy $\operatorname{dist}(K_{\eps,k},K_{\eps,k'})\ge \ell_\eps-2R_Ka_\eps\ge \ell_\eps/2$ for sufficiently small $\eps$. Each inclusion is therefore confined to the interior of its own cell, and the connected exterior of the reference obstacle leaves neighboring fluid-cell regions connected through the cell faces. Moreover,
\[
 \#\Ie\leq C\ell_\eps^{-3},\quad
 |\He|\leq C\ell_\eps^6,\quad
 \Big|\Omega\setminus\bigcup_{k\in\Ie}Q_{\eps,k}\Big|
 \leq C\ell_\eps.
\]
The last estimate concerns the boundary strip in a bounded domain. On the torus the cells cover $\Omega$ up to their faces, which is a measure-zero set. Cell sums below range over $k\in\Ie$.

The ambient velocity spaces are
\[
 \W=\begin{cases}
 H^1_0(\Omega;\R^3),&\Omega\subset\R^3,\\
 H^1(\T;\R^3),&\Omega=\T,
 \end{cases}
 \quad \cV=\{\vv\in\W:\Div \vv=0\},\quad
 \cH=\overline \cV^{\,L^2(\Omega;\R^3)}.
\]
Let $\W_\eps$ be $H^1_0(\Oe;\R^3)$ in a bounded domain and the space of periodic $H^1(\Oe;\R^3)$ fields with zero trace on every hole on the torus. Set
\[
 \cV_\eps=\{\vv\in\W_\eps:\Div \vv=0\},\quad
 \cH_\eps=\overline{\cV_\eps}^{\,L^2(\Oe;\R^3)}.
\]
The zero extension of a vector field $\vv$ on $\Oe$ is defined by
\[
 \widetilde \vv(x)=
 \begin{cases}\vv(x),&x\in\Oe,\\0,&x\in\He.\end{cases}
\]
Zero extension maps $\cV_\eps$ isometrically into $\cV$ in the $H^1$ norm and, by closure, maps $\cH_\eps$ into $\cH$. Thus the initial $L^2$ velocity is incorporated through the closure defining $\cH_\eps$. For scalar fields or their gradients defined only in $\Oe$, the notation $\chi_\eps z$ likewise means zero filling. This operation is distinct from the scalar Sobolev extension in \cref{subsec:extension}.

The smooth solenoidal test space is
\[
 \Ssig=\begin{cases}
 \{\vpsi\in C_c^\infty(\Omega;\R^3):\Div\vpsi=0\},
       &\Omega\subset\R^3,\\
 \{\vpsi\in C^\infty(\T;\R^3):\Div\vpsi=0\},
       &\Omega=\T.
 \end{cases}
\]
In the toroidal case all fields are periodic and constant vector fields belong to $\Ssig$. The standard solenoidal density theorem gives density of $\Ssig$ in $\cV$ for the $H^1$ norm, and hence in $\cH$ for the $L^2$ norm. Smooth time-dependent tests in a bounded domain have a common compact spatial support.

\subsection{Microscopic model and constitutive assumptions}\label{subsec:model}
The microscopic evolution is given by \eqref{eq:micro} on the domains \eqref{eq:geometry}. For fixed constants $0<\nu_*\le\nu^*$ and $0<m_*\le m^*$, the constitutive functions satisfy
\begin{equation}\label{eq:constitutive}
 \begin{gathered}
 F(s)=\tfrac14(s^2-1)^2,\qquad s\in\R,\\
 \nu\in C^1(\R),\qquad
 \nu_*\leq\nu(s)\leq\nu^*\ \text{for }s\in\R,
 \qquad \|\nu'\|_{L^\infty(\R)}<\infty,\\
 m\in C^{0,1}(\R),\qquad
 m_*\leq m(s)\leq m^*\ \text{for }s\in\R.
 \end{gathered}
\end{equation}
These constants are independent of $\eps$, and $C^{0,1}(\R)$ denotes the space of globally Lipschitz continuous functions on $\R$. The capillary coefficient $\lambda_\eps>0$ is independent of time and space.
Let $\vn$ be the outward unit normal of the fluid domain and write $\partial_n=\vn\cdot\nabla$. On every solid boundary we impose
\begin{equation}\label{eq:boundary}
 \vu_\eps=0,\quad \partial_n\phi_\eps=0,\quad
 m(\phi_\eps)\partial_n\mu_\eps=0.
\end{equation}
These conditions also hold at the outer wall in the bounded-domain case and all fields are periodic in the toroidal case. The scalar conditions are understood variationally. The interfacial length in the chemical-potential equation is fixed and has been normalized to one.

The quartic potential is used through the bound $F'(\phi_\eps)\in L^\infty(0,T;L^2(\Oe))$, which follows from the phase-energy bound and the uniform Sobolev inequality proved below. This makes the phase-difference test in the chemical-potential identity available at the energy level. The $C^1$ assumption on $\nu$ is used to pass to the limit in its Sobolev derivative, whereas the mobility enters the compactness argument through bounded continuity. 

Throughout the paper, the Newtonian stress is written as $2\nu(\phi_\eps)D\vu_\eps$. For comparison with the alternative convention in which the same stress is written as $\nu_{\mathrm{ref}}(\phi_\eps)D\vu_\eps$, we set $\nu_{\mathrm{ref}}:=2\nu$. Thus $\nu_{\mathrm{ref}}$ is an alternative normalization of the same viscosity coefficient. \Cref{cor:spheres} gives the spherical-obstacle resistance in both conventions.

For $G=\Omega$ or $G=\Oe$, $\lambda>0$, $\vv\in L^2(G;\R^3)$ and $\varphi\in H^1(G)$, define
\[
 \calF_G(\varphi)=\int_G\left(\tfrac12|\nabla\varphi|^2+F(\varphi)\right)\dx,
 \quad
 \calE_{G,\lambda}(\vv,\varphi)=\tfrac12\|\vv\|_{L^2(G)}^2+\lambda\calF_G(\varphi).
\]

Suppose the initial velocity and phase satisfy $(\vu_\eps^0,\phi_\eps^0)\in \cH_\eps\times H^1(\Oe)$ and the force satisfies $\vf_\eps\in L^2(0,T;L^2(\Oe;\R^3))$.

\subsection{The scalar extension operator}\label{subsec:extension}
The phase compactness is formulated on the fixed ambient domain $\Omega$. Velocity fields are handled by zero extension, whereas the scalar $H^1$ estimates require an extension across the hole interiors that is uniform in $\eps$. Accordingly, we construct a scalar extension operator and record the local estimates used later in the chemical-potential argument.

Set $B_0=B(0,1/2)$ and $U=B_0\setminus K$. Choose a bounded linear Sobolev extension $\mathfrak T:H^1(U)\to H^1(B_0)$ satisfying $(\mathfrak T w)|_U=w$. Define its constant-preserving version by
\begin{equation}\label{eq:referenceextension}
 \mathfrak E w=\mathfrak T\big(w-(w)_U\big)+(w)_U,
 \quad w\in H^1(U).
\end{equation}
For $z\in H^1(\Oe)$, let $z_{\eps,k}(y)=z(x_{\eps,k}+a_\eps y)$ for $y\in U$. We define
\begin{equation}\label{eq:extensiondefinition}
 \Eeps:H^1(\Oe)\to H^1(\Omega),\quad
 (\Eeps z)(x)=
 \begin{cases}
 z(x),&x\in\Oe,\\
 (\mathfrak E z_{\eps,k})\big((x-x_{\eps,k})/a_\eps\big),
      &x\in K_{\eps,k},\ k\in\Ie.
 \end{cases}
\end{equation}
The traces agree at each hole boundary. The operator is fixed in time and acts on time-dependent functions pointwise in time. Let
\[
 \calA_\eps=\bigcup_{k\in\Ie}
       \big(B(x_{\eps,k},a_\eps/2)\setminus K_{\eps,k}\big).
\]
These neighborhoods are disjoint and $|\calA_\eps|\leq C\ell_\eps^6$.

The local construction provides uniform scalar inequalities and controls the gradient energy assigned to the holes by the extension.
\begin{lemma}[Local extension and scalar inequalities]\label{lem:extension}
The linear operator defined by \eqref{eq:extensiondefinition} preserves constants,
restricts to the identity on $\Oe$,
and satisfies
\begin{align}
 \|\Eeps z\|_{L^2(\Omega)}^2
 &\leq C\big(\|z\|_{L^2(\Oe)}^2+a_\eps^2\|\nabla z\|_{L^2(\Oe)}^2\big),\label{eq:extensionL2}\\
 \|\nabla \Eeps z\|_{L^2(\Omega)}&\leq C\|\nabla z\|_{L^2(\Oe)},\notag\\
 \int_{\He}|\nabla \Eeps z|^2\dx
 &\leq C\int_{\calA_\eps}|\nabla z|^2\dx.\label{eq:extensionlocal}
\end{align}
In particular, $\|\Eeps z\|_{H^1(\Omega)}\leq C\|z\|_{H^1(\Oe)}$. Uniformly in $\eps$,
\begin{equation}\label{eq:scalarSP}
 \|z\|_{L^6(\Oe)}\leq C\|z\|_{H^1(\Oe)},\quad
 \|z-(z)_{\Oe}\|_{L^2(\Oe)}\leq C\|\nabla z\|_{L^2(\Oe)}.
\end{equation}
\end{lemma}
\begin{proof}
The reference fluid set $U$ is connected and Lipschitz. Put $w_0=w-(w)_U$. Since $w_0=0$ for constant $w$, formula \eqref{eq:referenceextension} shows at once that $\mathfrak E$ preserves constants. The boundedness of $\mathfrak T$ and the Poincar\'e inequality on this fixed set give
\[
 \|w_0\|_{H^1(U)}\leq C\|\nabla w\|_{L^2(U)},\qquad
 \|\nabla\mathfrak E w\|_{L^2(B_0)}
 =\|\nabla\mathfrak T w_0\|_{L^2(B_0)}
 \leq C\|\nabla w\|_{L^2(U)}.
\]
Moreover, $|(w)_U|\leq |U|^{-1/2}\|w\|_{L^2(U)}$, so
\[
 \|\mathfrak E w\|_{L^2(B_0)}^2
 \leq C\big(\|w\|_{L^2(U)}^2+\|\nabla w\|_{L^2(U)}^2\big).
\]
Write $U_{\eps,k}=x_{\eps,k}+a_\eps U$. The change of variables
$x=x_{\eps,k}+a_\eps y$ gives
\[
 \begin{aligned}
 \int_{K_{\eps,k}}|\Eeps z|^2\dx
 &\leq Ca_\eps^3
     \big(\|z_{\eps,k}\|_{L^2(U)}^2
               +\|\nabla_yz_{\eps,k}\|_{L^2(U)}^2\big)=C\left(\int_{U_{\eps,k}}|z|^2\dx
            +a_\eps^2\int_{U_{\eps,k}}|\nabla z|^2\dx\right),\\
 \int_{K_{\eps,k}}|\nabla\Eeps z|^2\dx
 &\leq Ca_\eps\|\nabla_yz_{\eps,k}\|_{L^2(U)}^2
 =C\int_{U_{\eps,k}}|\nabla z|^2\dx.
 \end{aligned}
\]
Since the neighborhoods $U_{\eps,k}$ are disjoint and the extension equals $z$ in $\Oe$, summing proves \eqref{eq:extensionL2}--\eqref{eq:extensionlocal}, together with the global gradient bound. The fixed-domain Sobolev embedding then gives
\[
 \|z\|_{L^6(\Oe)}\leq\|\Eeps z\|_{L^6(\Omega)}
 \leq C\|\Eeps z\|_{H^1(\Omega)}
 \leq C\|z\|_{H^1(\Oe)}.
\]

For the mean-zero estimate, suppose no common Poincar\'e constant exists for all sufficiently small $\eps$. There are then $\eps_j\to0$ and $z_j\in H^1(\Omega_{\eps_j})$ such that
\[
 \int_{\Omega_{\eps_j}}z_j\dx=0,\qquad
 \|z_j\|_{L^2(\Omega_{\eps_j})}=1,\qquad
 \|\nabla z_j\|_{L^2(\Omega_{\eps_j})}\to0.
\]
The extension estimates already proved imply that
$Z_j=\mathscr E_{\eps_j}z_j$ is bounded in $H^1(\Omega)$ and
$\nabla Z_j\to0$ in $L^2(\Omega)$. Rellich compactness on the connected ambient domain yields, after extraction, $Z_j\to c$ in $L^2(\Omega)$ for a constant $c$. The mean-zero condition identifies this constant:
\[
 \left|\int_\Omega Z_j\dx\right|
 =\left|\int_{\mathcal H_{\eps_j}}Z_j\dx\right|
 \leq |\mathcal H_{\eps_j}|^{1/2}\|Z_j\|_{L^2(\Omega)}\to0.
\]
Consequently $c|\Omega|=0$ and $c=0$. On the other hand,
\[
 \|Z_j\|_{L^2(\Omega)}\geq
 \|z_j\|_{L^2(\Omega_{\eps_j})}=1,
\]
which contradicts $Z_j\to0$ in $L^2(\Omega)$. This proves the second estimate in \eqref{eq:scalarSP}.
\end{proof}

Two consequences of the construction will be used below. For $z\in H^1(\Omega)$, set $d_\eps=\Eeps(z|_{\Oe})-z$. The difference is supported in $\He$ and has zero trace on each hole boundary. The local estimate and the rescaled Dirichlet Poincar\'e inequality give
\begin{equation}\label{eq:fixedscalarrecovery}
 \|\nabla d_\eps\|_2^2\leq C\int_{\calA_\eps\cup\He}|\nabla z|^2\dx, \quad \|d_\eps\|_2\leq Ca_\eps\|\nabla d_\eps\|_2.
\end{equation}
Absolute continuity of the integral implies $\Eeps(z|_{\Oe})\to z$ in $H^1(\Omega)$. For a microscopic scalar $z\in H^1(\Oe)$, the difference between Sobolev extension and zero filling satisfies
\begin{equation}\label{eq:fillingerror}
 \|\Eeps z-\chi_\eps z\|_{L^2(\Omega)}
 \leq |\He|^{1/3}\|\Eeps z\|_{L^6(\Omega)}
 \leq C\ell_\eps^2\|z\|_{H^1(\Oe)}.
\end{equation}
Thus $\Eeps z$ and $\chi_\eps z$ agree on $\Oe$ and differ only inside the holes, with the $L^2(\Omega)$ discrepancy quantified by \eqref{eq:fillingerror}. Consequently, time compactness of the zero-filled phase transfers to the Sobolev extension while the latter retains the uniform spatial $H^1$ bounds.

\subsection{Finite-energy weak solution}\label{subsec:weak}
With the fixed-domain scalar extension available, we now specify the microscopic energy solution class used throughout the homogenization argument.

\begin{definition}[Finite-energy weak solution]\label{def:weak}
A triple $(\vu_\eps,\phi_\eps,\mu_\eps)$ is called a finite-energy weak solution of \eqref{eq:micro} with the boundary conditions \eqref{eq:boundary} if it has the regularity

\[
 \begin{gathered}
 \vu_\eps\in L^\infty(0,T;\cH_\eps)\cap L^2(0,T;\cV_\eps)
                   \cap C_w([0,T];\cH_\eps),\\
 \phi_\eps\in L^\infty(0,T;H^1(\Oe))\cap C([0,T];L^2(\Oe)),
 \quad \mu_\eps\in L^2(0,T;H^1(\Oe)),\\
 \partial_t\vu_\eps\in L^{4/3}(0,T;\cV_\eps'),\quad
 \partial_t\phi_\eps\in L^2(0,T;(H^1(\Oe))');
 \end{gathered}
\]
for $\vpsi\in C^1([0,T];\cV_\eps)$ with $\vpsi(T)=0$, it satisfies
\[
 \begin{aligned}
 -\int_0^T\!\int_\Oe \vu_\eps\cdot\partial_t\vpsi\dx\dt
 -\int_0^T\!\int_\Oe(\vu_\eps\otimes \vu_\eps):\nabla\vpsi\dx\dt
 +\int_0^T\!\int_\Oe2\nu(\phi_\eps)D\vu_\eps:D\vpsi\dx\dt\\
 =\int_0^T\!\int_\Oe(\vf_\eps-\lambda_\eps\phi_\eps\nabla\mu_\eps)\cdot\vpsi\dx\dt
  +\int_\Oe \vu_\eps^0\cdot\vpsi(0)\dx;
 \end{aligned}
\]
for $\zeta\in C^1([0,T];H^1(\Oe))$ with $\zeta(T)=0$, and for $\eta\in L^2(0,T;H^1(\Oe))$, respectively,
\begin{equation}\label{eq:weakCH}
 \begin{aligned}
 &-\int_0^T\!\int_\Oe\phi_\eps\partial_t\zeta\dx\dt
 -\int_0^T\!\int_\Oe\phi_\eps \vu_\eps\cdot\nabla\zeta\dx\dt\\
 &\quad+\int_0^T\!\int_\Oe m(\phi_\eps)\nabla\mu_\eps\cdot\nabla\zeta\dx\dt
 =\int_\Oe\phi_\eps^0\zeta(0)\dx,
 \end{aligned}
\end{equation}
\begin{equation}\label{eq:weakchemical}
 \int_0^T\!\int_\Oe\mu_\eps\eta\dx\dt
 =\int_0^T\!\int_\Oe
       \big(\nabla\phi_\eps\cdot\nabla\eta+F'(\phi_\eps)\eta\big)\dx\dt;
\end{equation}
the initial traces agree with $(\vu_\eps^0,\phi_\eps^0)$, and for almost every $t\in(0,T)$,
\[
 \begin{aligned}
 \calE_{\Oe,\lambda_\eps}(\vu_\eps(t),\phi_\eps(t))
 +\int_0^t\!\int_\Oe
   \big(2\nu(\phi_\eps)|D\vu_\eps|^2
       +\lambda_\eps m(\phi_\eps)|\nabla\mu_\eps|^2\big)\dx\dd s\\
 \leq\calE_{\Oe,\lambda_\eps}(\vu_\eps^0,\phi_\eps^0)
       +\int_0^t\!\int_\Oe \vf_\eps\cdot \vu_\eps\dx\dd s.
 \end{aligned}
\]
\end{definition}

\begin{remark}
The scalar test functions have unrestricted traces, so the homogeneous Neumann conditions for \(\phi_\varepsilon\) and \(\mu_\varepsilon\) are encoded by the absence of boundary terms in \eqref{eq:weakCH}--\eqref{eq:weakchemical}. In a bounded domain these scalar conditions hold on the whole boundary \(\partial\Omega_\varepsilon\), including the outer wall and the velocity satisfies the no-slip condition there through \(\mathcal W_\varepsilon=H^1_0(\Omega_\varepsilon;\mathbb R^3)\). On the torus, the velocity is periodic with no-slip trace on the holes.
A phase with the regularity in \cref{def:weak} has a representative in $C_w([0,T];H^1(\Oe))$: boundedness in $H^1$ and continuity in $L^2$ identify every weak $H^1$ time limit. In particular, its $H^1$ bound may be taken at every time. The momentum identity is written with solenoidal test fields and the corresponding macroscopic pressure is reconstructed in \cref{sec:limit}.
\end{remark}

\subsection{Stokes capacity and the coupled limit}\label{subsec:main}
The critical geometry enters the limit through the Stokes capacity of the reference obstacle. Let $\ve_1,\ve_2,\ve_3$ be the standard basis of $\R^3$. For $i=1,2,3$, let $(\vh^i,q^i)$ be the finite-energy exterior Stokes velocity--pressure pair, with $\vh^i\in L^6(\R^3\setminus K;\R^3)$ and $\nabla \vh^i\in L^2(\R^3\setminus K;\R^{3\times3})$, satisfying
\begin{equation}\label{eq:exterior}
 \left\{\begin{aligned}
 &-\Delta \vh^i+\nabla q^i=0,\quad \Div \vh^i=0 &&\text{in }\R^3\setminus K,\\
& \vh^i=\ve_i&&\text{on }\partial K,\\
& \vh^i(x)\to0&&\text{as }|x|\to\infty.
 \end{aligned}\right.
\end{equation}
The additive constant in $q^i$ is fixed by decay at infinity. The unit-viscosity Stokes capacity matrix and the macroscopic resistance matrix are
\begin{equation}\label{eq:capacity}
 \M_{ij}=\int_{\R^3\setminus K}2D\vh^i:D\vh^j\dx,
 \quad \B=\beta\M.
\end{equation}
The matrix $\M$ is symmetric and positive definite, as verified in \cref{sec:correctors}. It is determined by the reference obstacle $K$ and the unit-viscosity normalization. The macroscopic cell scale enters through the factor $\beta$ in $\B=\beta\M$.

We assume that, for some $\vu^0\in \cH$, $\phi^0\in H^1(\Omega)$ and $\vf\in L^2_tL^2_x$,
\begin{equation}\label{eq:data}
 \widetilde \vu_\eps^0\to \vu^0\text{ in }L^2(\Omega),\quad
 \Eeps\phi_\eps^0\to\phi^0\text{ in }H^1(\Omega),\quad
 \widetilde \vf_\eps\to \vf\text{ in }L^2_tL^2_x.
\end{equation}
These assumptions imply convergence of the initial energies when $\lambda_\eps\to\lambda>0$. Indeed, with $\Phi_\eps^0=\Eeps\phi_\eps^0$, strong $H^1(\Omega)$ convergence gives strong $L^4(\Omega)$ convergence and
\[
 \begin{aligned}
 \int_{\He}|\nabla\Phi_\eps^0|^2\dx
 &\leq 2\int_{\He}|\nabla\Phi_\eps^0-\nabla\phi^0|^2\dx
      +2\int_{\He}|\nabla\phi^0|^2\dx\\
 &\leq 2\|\nabla\Phi_\eps^0-\nabla\phi^0\|_{L^2(\Omega)}^2
      +2\int_{\He}|\nabla\phi^0|^2\dx\to0.
 \end{aligned}
\]
Also $F(\Phi_\eps^0)\to F(\phi^0)$ in $L^1(\Omega)$, so its integral over $\He$ vanishes. Therefore
\[
 \calE_{\Oe,\lambda_\eps}(\vu_\eps^0,\phi_\eps^0)
 \to\calE_{\Omega,\lambda}(\vu^0,\phi^0).
\]
These assumptions are compatible with arbitrary data in the stated limiting spaces. By \eqref{eq:fixedscalarrecovery}, the phase $\phi^0$ can be restricted to $\Oe$. For the velocity, a smooth solenoidal approximation followed by the correction constructed in \cref{sec:correctors} yields admissible initial data by a diagonal choice.

Write $(\vu,\phi,\mu)$ for the prospective limit variables and let $p$ denote the macroscopic pressure. The capacity matrix in \eqref{eq:capacity} leads to the effective NSCH--Brinkman system
\begin{equation}\label{eq:limit}
 \left\{\begin{aligned}
 \partial_t\vu+\Div(\vu\otimes \vu)-\Div(2\nu(\phi)D\vu)
       +\nu(\phi)\B \vu+\nabla p&=\vf-\lambda\phi\nabla\mu,\\
 \Div \vu&=0,\\
 \partial_t\phi+\Div(\phi \vu)&=\Div(m(\phi)\nabla\mu),\\
 \mu&=-\Delta\phi+F'(\phi).
 \end{aligned}\right.
\end{equation}
The solenoidal weak formulation will be stated later and the pressure is reconstructed from that formulation in \cref{sec:limit}.

\begin{definition}[Finite-energy weak solution of the effective system]\label{def:effectiveweak}
Let \(\lambda>0\), \(\vu^0\in\cH\), \(\phi^0\in H^1(\Omega)\) and \(\vf\in L^2(0,T;L^2(\Omega;\mathbb R^3))\). A triple \((\vu,\phi,\mu)\) is a finite-energy weak solution of the effective NSCH--Brinkman system if
\[
\begin{gathered}
 \vu\in L^\infty(0,T;\cH)\cap L^2(0,T;\cV)\cap C_w([0,T];\cH),\\
 \phi\in L^\infty(0,T;H^1(\Omega))\cap C([0,T];L^2(\Omega)),
 \qquad \mu\in L^2(0,T;H^1(\Omega)),\\
 \partial_t\vu\in L^{4/3}(0,T;\cV'),\qquad
 \partial_t\phi\in L^2(0,T;(H^1(\Omega))').
\end{gathered}
\]
For any \(\vpsi\in C^1([0,T];\cV)\) with \(\vpsi(T)=0\),
\[
 \begin{aligned}
 &-\int_{\Dt}\vu\cdot\partial_t\vpsi\dx\dt
  -\int_{\Dt}(\vu\otimes\vu):\nabla\vpsi\dx\dt+\int_{\Dt}
       \big(2\nu(\phi)D\vu:D\vpsi+\nu(\phi)\B\vu\cdot\vpsi\big)\dx\dt\\
 &=\int_{\Dt}(\vf-\lambda\phi\nabla\mu)\cdot\vpsi\dx\dt
     +\int_\Omega\vu^0\cdot\vpsi(0)\dx.
 \end{aligned}
\]
For any \(\zeta\in C^1([0,T];H^1(\Omega))\) with \(\zeta(T)=0\) and any \(\eta\in L^2(0,T;H^1(\Omega))\), the scalar variables satisfy
\[
\begin{aligned}
 &-\int_{\Dt}\phi\,\partial_t\zeta\,\dx\dt
 -\int_{\Dt}\phi\vu\cdot\nabla\zeta\,\dx\dt+\int_{\Dt}m(\phi)\nabla\mu\cdot\nabla\zeta\,\dx\dt
 =\int_\Omega\phi^0\zeta(0)\,\dx,
\end{aligned}
\]
\[
\int_{\Dt}\mu\eta\,\dx\dt
 =\int_{\Dt}\big(\nabla\phi\cdot\nabla\eta+F'(\phi)\eta\big)\,\dx\dt.
 \]
In a bounded domain these scalar identities encode the homogeneous Neumann conditions for \(\phi\) and \(\mu\) in the variational sense: the scalar test functions have unrestricted boundary traces and no boundary functional appears. The initial traces equal \((\vu^0,\phi^0)\), and for almost every \(t\in(0,T)\),
\begin{equation}\label{eq:limitenergy}
 \begin{aligned}
 \calE_{\Omega,\lambda}(\vu(t),\phi(t))
 +\int_0^t\!\int_\Omega
 \big(2\nu(\phi)|D\vu|^2+\nu(\phi)\B\vu\cdot\vu
       +\lambda m(\phi)|\nabla\mu|^2\big)\dx\dd s\\
 \leq\calE_{\Omega,\lambda}(\vu^0,\phi^0)
       +\int_0^t\!\int_\Omega\vf\cdot\vu\dx\dd s.
 \end{aligned}
\end{equation}
\end{definition}

For every fixed sufficiently small \(\varepsilon\), the microscopic finite-energy solution class is nonempty by \cref{prop:existence}. With the microscopic and effective solution classes fixed, the critical homogenization result can now be stated.

\begin{theorem}[Critical NSCH homogenization]\label{thm:main}
Let the geometry and constitutive functions satisfy \crefrange{subsec:geometry}{subsec:model}. Assume \eqref{eq:data} and $\lambda_\eps\to\lambda\in(0,\infty)$. Let $(\vu_\eps,\phi_\eps,\mu_\eps)$ be finite-energy weak solutions on a fixed interval $[0,T]$.
As $\eps\to0$, up to a subsequence,
\begin{equation}\label{eq:mainconvergences}
 \begin{aligned}
 \widetilde \vu_\eps&\wstar \vu&&\text{in }L^\infty(0,T;\cH),
 &\widetilde \vu_\eps&\weak \vu&&\text{in }L^2(0,T;\cV),\\
 \widetilde \vu_\eps&\to \vu&&\text{in }L^2_tL^2_x,
 &\Eeps\phi_\eps&\to\phi&&\text{in }L^2(0,T;H^1(\Omega)),\\
 \Eeps\phi_\eps&\wstar\phi&&\text{in }L^\infty(0,T;H^1(\Omega)),
 &\Eeps\mu_\eps&\weak\mu&&\text{in }L^2(0,T;H^1(\Omega)).
 \end{aligned}
\end{equation}
Moreover,
\[
 \nu(\Eeps\phi_\eps)\to\nu(\phi)\quad\text{strongly in }L^2(0,T;H^1(\Omega)).
\]

The limit \((\vu,\phi,\mu)\) is a finite-energy weak solution in the sense of \cref{def:effectiveweak}. There exists a pressure distribution \(p\) such that \((\vu,\phi,\mu,p)\) satisfies \eqref{eq:limit} with initial data $(\vu^0,\phi^0)$. In a bounded domain, $\vu=0$ on $\partial\Omega$ and the homogeneous scalar Neumann conditions hold in the variational sense specified in \cref{def:effectiveweak}. On $\T$, all fields are periodic. The phase mass is conserved. Along the same subsequence,
\begin{equation}\label{eq:strongtimetraces}
 \Eeps\phi_\eps\to\phi\quad\text{in }C([0,T];L^2(\Omega)),\quad
 \sup_{t\in[0,T]}| (\widetilde \vu_\eps(t)-\vu(t),\vz)_\Omega |\to0
 \quad(\vz\in \cH).
\end{equation}

A pressure primitive may be chosen so that
\[
 P\in L^\infty(0,T;L^2_0(\Omega)),\qquad
 p=\partial_tP\quad\text{in }\mathcal D'((0,T)\times\Omega).
\]
\end{theorem}

The additional term in \eqref{eq:limit} is local in the phase: the scalar factor is $\nu(\phi(t,x))$, while $\B$ is determined by the obstacle geometry. The mobility and the chemical-potential operator are unchanged. The theorem applies on every finite time interval. For spherical holes, the resistance becomes explicit.
\begin{corollary}[Spherical obstacles]\label{cor:spheres}
If $K=\overline{B(0,r)}$, $0<r<1/8$, then $\M=6\pi r\Ithree$. The Brinkman term in \eqref{eq:limit} is $6\pi\beta r\nu(\phi)\vu$. With the alternative microscopic stress $\nu_{\mathrm{ref}}(\phi)D\vu$, the term is $3\pi\beta r\nu_{\mathrm{ref}}(\phi)\vu$.
\end{corollary}

\subsection{A vanishing-capillarity limit}\label{subsec:vanishingstatement}

For the second regime, let $\lambda_\eps\to0$ and normalize the energy by $\lambda_\eps$. The data are prescribed so that the kinetic and phase energies remain bounded at this scale. In particular, the convergence assumptions below imply
\[
 \sup_\eps\left(\|\vv_\eps^0\|_{L^2(\Oe)}^2+\calF_{\Oe}(\phi_\eps^0)\right)<\infty.
\]
Suppose $\lambda_\eps>0$, $\lambda_\eps\to0$, and set
\begin{equation}\label{eq:normalization}
 \vv_\eps=\lambda_\eps^{-1/2}\vu_\eps,\quad
 \vv_\eps^0=\lambda_\eps^{-1/2}\vu_\eps^0,\quad
 \vg_\eps=\lambda_\eps^{-1/2}\vf_\eps.
\end{equation}
For $\vv^0\in \cH$, $\phi^0\in H^1(\Omega)$ and $\vg\in L^2_tL^2_x$, assume
\begin{equation}\label{eq:vanishingdata}
 \widetilde \vv_\eps^0\to \vv^0\text{ in }L^2(\Omega),\quad
 \Eeps\phi_\eps^0\to\phi^0\text{ in }H^1(\Omega),\quad
 \widetilde{\vg}_\eps\to \vg\text{ in }L^2_tL^2_x.
\end{equation}
The energy inequality divided by $\lambda_\eps$ gives uniform bounds for $(\vv_\eps,\phi_\eps,\mu_\eps)$.

The normalized limit retains the time derivative in the momentum equation, while the transport and capillary forces vanish. Write $(\vv,\phi,\mu)$ for the corresponding limit variables and let $\pi$ denote the pressure distribution. The formal limiting system is
\begin{equation}\label{eq:vanishinglimit}
 \left\{\begin{aligned}
 \partial_t\vv-\Div(2\nu(\phi)D\vv)+\nu(\phi)\B \vv+\nabla\pi&=\vg,\\
 \Div \vv&=0,\\
 \partial_t\phi&=\Div(m(\phi)\nabla\mu),\\
 \mu&=-\Delta\phi+F'(\phi).
 \end{aligned}\right.
\end{equation}
The energy solution class associated with this system is specified next.

\begin{definition}[Finite-energy weak solution of the unsteady Stokes--Brinkman/Cahn--Hilliard system]\label{def:vanishingweak}
Let \(\vv^0\in\cH\), \(\phi^0\in H^1(\Omega)\) and \(\vg\in L^2(0,T;L^2(\Omega;\mathbb R^3))\). A triple \((\vv,\phi,\mu)\) is a finite-energy weak solution of the unsteady Stokes--Brinkman/Cahn--Hilliard system if
\[
\begin{gathered}
 \vv\in L^\infty(0,T;\cH)\cap L^2(0,T;\cV)\cap C([0,T];\cH),\\
 \phi\in L^\infty(0,T;H^1(\Omega))\cap C([0,T];L^2(\Omega)),
 \qquad \mu\in L^2(0,T;H^1(\Omega)),\\
 \partial_t\vv\in L^2(0,T;\cV'),\qquad
 \partial_t\phi\in L^2(0,T;(H^1(\Omega))').
\end{gathered}
\]
For any \(\vpsi\in C^1([0,T];\cV)\) with \(\vpsi(T)=0\),
\[
 -\int_{\Dt}\vv\cdot\partial_t\vpsi\,\mathrm dx\,\mathrm dt
 +\int_{\Dt}\big(2\nu(\phi)D\vv:D\vpsi+\nu(\phi)\B\vv\cdot\vpsi\big)\,\mathrm dx\,\mathrm dt
 =\int_{\Dt}\vg\cdot\vpsi\,\mathrm dx\,\mathrm dt
 +\int_\Omega\vv^0\cdot\vpsi(0)\,\mathrm dx.
\]
For any scalar \(\zeta\in C^1([0,T];H^1(\Omega))\) with \(\zeta(T)=0\) and any \(\eta\in L^2(0,T;H^1(\Omega))\),
\[
 -\int_{\Dt}\phi\,\partial_t\zeta\,\mathrm dx\,\mathrm dt
 +\int_{\Dt}m(\phi)\nabla\mu\cdot\nabla\zeta\,\mathrm dx\,\mathrm dt
 =\int_\Omega\phi^0\zeta(0)\,\mathrm dx,
\]
\[
 \int_{\Dt}\mu\eta\,\mathrm dx\,\mathrm dt
 =\int_{\Dt}\big(\nabla\phi\cdot\nabla\eta+F'(\phi)\eta\big)\,\mathrm dx\,\mathrm dt.
\]
The initial traces equal \((\vv^0,\phi^0)\), and for almost every \(t\in(0,T)\),
\[
 \begin{aligned}
 \tfrac12\|\vv(t)\|_2^2+\calF_\Omega(\phi(t))
 +\int_0^t\!\int_\Omega
 \big(2\nu(\phi)|D\vv|^2+\nu(\phi)\B\vv\cdot\vv
       +m(\phi)|\nabla\mu|^2\big)\,\mathrm dx\,\mathrm ds\\
 \leq\tfrac12\|\vv^0\|_2^2+\calF_\Omega(\phi^0)
      +\int_0^t\!\int_\Omega\vg\cdot\vv\,\mathrm dx\,\mathrm ds.
 \end{aligned}
\]
The boundary conditions are the no-slip condition for \(\vv\), the variational homogeneous Neumann conditions for \(\phi\) and \(\mu\) in a bounded domain and all fields are periodic on \(\mathbb T^3\).
\end{definition}

With the rescaled limit system and its energy class specified, we state the second homogenization theorem.
\begin{theorem}[Unsteady Stokes--Brinkman limit]\label{thm:vanishing}
Let the geometry and constitutive functions satisfy \crefrange{subsec:geometry}{subsec:model}. Suppose $\lambda_\eps>0$, $\lambda_\eps\to0$ and \eqref{eq:vanishingdata} holds for the normalization \eqref{eq:normalization}. Then finite-energy weak solutions have a subsequence such that
\[
 \begin{aligned}
 \widetilde\vv_\eps&\wstar\vv
    &&\text{in }L^\infty(0,T;\cH),\\
 \widetilde\vv_\eps&\weak\vv
    &&\text{in }L^2(0,T;\cV),\\
 \widetilde\vv_\eps&\to\vv
    &&\text{in }L^2(0,T;\cH),\\
 \Eeps\phi_\eps&\to\phi
    &&\text{in }L^2(0,T;H^1(\Omega))\cap C([0,T];L^2(\Omega)),\\
 \Eeps\phi_\eps&\wstar\phi
    &&\text{in }L^\infty(0,T;H^1(\Omega)),\\
 \Eeps\mu_\eps&\weak\mu
    &&\text{in }L^2(0,T;H^1(\Omega)).
 \end{aligned}
\]
For any $\vz\in\cH$,
\[
 \sup_{t\in[0,T]}
 |(\widetilde\vv_\eps(t)-\vv(t),\vz)_\Omega|\to0.
\]
There exists a pressure distribution $\pi$ such that $(\vv,\phi,\mu,\pi)$ satisfies \eqref{eq:vanishinglimit} with initial data $(\vv^0,\phi^0)$. Moreover, $(\vv,\phi,\mu)$ is a finite-energy weak solution in the sense of \cref{def:vanishingweak}. In particular,
\[
 \partial_t\vv\in L^2(0,T;\cV'),\qquad \vv\in C([0,T];\cH),
\]
and the energy inequality in \cref{def:vanishingweak} holds for almost every time. In the original variables, $\widetilde \vu_\eps\to0$ strongly in $L^2(0,T;H^1(\Omega))$ and in $L^\infty(0,T;L^2(\Omega))$. The limiting momentum equation remains unsteady under the scaling \eqref{eq:normalization}.
\end{theorem}

\section{Uniform estimates and compactness of the phase}\label{sec:phase}
\subsection{Energy estimates and the mean chemical potential}
Throughout this section, we use the assumptions of \cref{thm:main}, especially the one that $\lambda_\eps$ is bounded above and away from zero. The energy inequality, Young's inequality and Gronwall's lemma yield
\begin{equation}\label{eq:uniformenergy}
 \|\widetilde \vu_\eps\|_{L^\infty_tL^2_x}
 +\|\widetilde \vu_\eps\|_{L^2_tH^1_x}
 +\|\Eeps\phi_\eps\|_{L^\infty_tH^1_x}
 +\|\nabla\mu_\eps\|_{L^2((0,T)\times\Oe)}\leq C_T.
\end{equation}
$F(s)$ controls $|s|^4$ up to an additive constant. The velocity gradient is controlled by the fixed-domain identity
\[
 \int_\Omega|\nabla\widetilde \vu_\eps|^2\dx
       =2\int_\Omega|D\widetilde \vu_\eps|^2\dx.
\]
To verify it, let $\boldsymbol w$ be a solenoidal field with either no-slip trace or periodic boundary conditions. Expanding $2|D\boldsymbol w|^2=|\nabla\boldsymbol w|^2+\partial_jw_i\,\partial_iw_j$ and integrating the cross term by parts gives zero because $\Div\boldsymbol w=0$ and either the trace vanishes or the field is periodic. In the periodic case the kinetic energy also controls the constant mode.

Testing \eqref{eq:weakchemical} by a spatial constant gives, almost everywhere in time,
\begin{equation}\label{eq:meanmu}
 (\mu_\eps)_{\Oe}=(F'(\phi_\eps))_{\Oe}.
\end{equation}
Since $F'(s)=s^3-s$, the uniform $L^6$ estimate implies
\begin{equation}\label{eq:muFbounds}
 \|F'(\phi_\eps)\|_{L^\infty(0,T;L^2(\Oe))}\leq C_T,
 \quad \|\Eeps\mu_\eps\|_{L^2_tH^1_x}\leq C_T.
\end{equation}
The second assertion follows from \eqref{eq:meanmu} and \eqref{eq:scalarSP}. Thus the chemical potential is controlled with its actual mean.

The phase flux and the mobility flux obey
\[
 \|\phi_\eps \vu_\eps\|_{L^2(0,T;L^2(\Oe))}
 +\|m(\phi_\eps)\nabla\mu_\eps\|_{L^2((0,T)\times\Oe)}\leq C_T.
\]
For the first term, use $\phi_\eps\in L^\infty_tL^3_x$ and $\vu_\eps\in L^2_tL^6_x$. On the fixed ambient domain, \eqref{eq:weakCH} consequently gives
\begin{equation}\label{eq:distributionderivative}
 \partial_t(\chi_\eps\phi_\eps)
 =-\Div\big((\Eeps\phi_\eps)\widetilde \vu_\eps\big)
  +\Div\big(\chi_\eps m(\phi_\eps)\nabla\mu_\eps\big),
 \quad
 \|\partial_t(\chi_\eps\phi_\eps)\|_{L^2(0,T;(H^1(\Omega))')}\leq C_T.
\end{equation}
The divergences in this formula act on all $H^1(\Omega)$ test functions. In particular, \eqref{eq:distributionderivative} encodes the no-flux condition on every solid boundary. These estimates provide the fixed-domain time control needed for the first compactness step.

\begin{lemma}[Compactness of the phase]\label{lem:phaseL2}
After subsequence extraction,
\[
 \begin{gathered}
 \Eeps\phi_\eps\to\phi\quad\text{in }C([0,T];L^2(\Omega)),\\
 \Eeps\phi_\eps\wstar\phi\quad\text{in }L^\infty_tH^1_x,\quad
 \Eeps\mu_\eps\weak\mu\quad\text{in }L^2_tH^1_x.
 \end{gathered}
\]
In particular, $\phi(0)=\phi^0$ in $L^2(\Omega)$.
\end{lemma}
\begin{proof}
First, $\Eeps\phi_\eps$ has a continuous $L^2(\Omega)$ representative because $\phi_\eps$ is weakly continuous in $H^1(\Oe)$, the operator $\Eeps$ is bounded between the corresponding $H^1$ spaces, and $H^1(\Omega)\hookrightarrow\hookrightarrow L^2(\Omega)$. For each fixed $\eps$, this gives a continuous $L^2(\Omega)$ representative.

Let $\{\xi_j\}_{j\geq1}\subset H^1(\Omega)$ be an orthonormal basis of $L^2(\Omega)$ consisting of Neumann eigenfunctions in the bounded case or real periodic Fourier modes on the torus. Define
\[
 \mathcal P_N^{\phi}:L^2(\Omega)\to\operatorname{span}\{\xi_1,\ldots,\xi_N\},
 \quad \mathcal P_N^{\phi}z=\sum_{j=1}^N(z,\xi_j)_\Omega\xi_j.
\]
Compactness of the embedding $H^1(\Omega)\hookrightarrow\hookrightarrow L^2(\Omega)$ gives numbers $\delta_N^{\phi}\to0$ such that
\[
 \|(\Id-\mathcal P_N^{\phi})z\|_2
       \leq\delta_N^{\phi}\|z\|_{H^1(\Omega)}.
\]
For every fixed mode, the coefficient of $\chi_\eps\phi_\eps$ is bounded in $H^1(0,T)$ by \eqref{eq:distributionderivative} and the energy bound, with a bound independent of $\eps$. After a diagonal extraction, all these coefficients converge in $C([0,T])$. Their differences from the coefficients of $\Eeps\phi_\eps$ are bounded uniformly in time by $C\ell_\eps^2$, by \eqref{eq:fillingerror}. Finally,
\[
 \sup_{t\in[0,T]}\|(\Id-\mathcal P_N^{\phi})\Eeps\phi_\eps(t)\|_2
       \leq C_T\delta_N^{\phi}.
\]
Finite-dimensional compactness and this uniform tail estimate show that the extensions are Cauchy in $C([0,T];L^2(\Omega))$. The weak convergences follow from the uniform bounds and \eqref{eq:data} identifies the initial value. This is a fixed-mode version of the compactness argument in \cite{Simon}, applied to the time derivative of the zero-filled phase.
\end{proof}

\subsection{Strong convergence of the phase gradients}
The critical viscous coefficient requires an $H^1$-level upgrade of the phase compactness. The variational chemical-potential identity provides this stronger convergence.

\begin{proposition}[Strong phase and viscosity compactness]\label{prop:strongphase}
Along the subsequence of \cref{lem:phaseL2},
\begin{equation}\label{eq:strongphase}
 \chi_\eps\nabla\phi_\eps\to\nabla\phi\quad\text{in }L^2(\Dt),
 \quad \Eeps\phi_\eps\to\phi\quad\text{in }L^2_tH^1_x.
\end{equation}
Furthermore,
\begin{equation}\label{eq:strongnonlinear}
 F'(\Eeps\phi_\eps)\to F'(\phi)\quad\text{in }L^2_tL^2_x,
 \quad
 \nu(\Eeps\phi_\eps)\to\nu(\phi)\quad\text{in }L^2_tH^1_x.
\end{equation}
The limiting chemical-potential identity holds for all $\eta\in L^2_tH^1_x$:
\begin{equation}\label{eq:limitchemical}
 \int_{\Dt}\mu\eta\dx\dt
   =\int_{\Dt}\big(\nabla\phi\cdot\nabla\eta+F'(\phi)\eta\big)\dx\dt.
\end{equation}
\end{proposition}
\begin{proof}
Write $\Phi_\eps=\Eeps\phi_\eps$. Weak $H^1$ convergence and the vanishing measure of the holes imply
\[
 \chi_\eps\nabla\phi_\eps\weak\nabla\phi\quad\text{in }L^2(\Dt).
\]
The pairing of \(\1_{\He}\nabla\Phi_\eps\) with a fixed \(L^2\) function tends to zero by absolute continuity of that test function's integral. The bounded sequence \(F'(\Phi_\eps)\) converges almost everywhere, after extraction, to \(F'(\phi)\) and hence converges weakly in \(L^2\). To pass from \(\Omega_\eps\) to the fixed domain, take first a smooth ambient test \(\eta\). The filling estimate \eqref{eq:fillingerror} applied to \(\mu_\eps\) gives
\[
 \|\Eeps\mu_\eps-\chi_\eps\mu_\eps\|_{L^2(\Dt)}
 \le C\ell_\eps^2\|\mu_\eps\|_{L^2(0,T;H^1(\Omega_\eps))}\to0.
\]
Moreover,
\[
 \left|\int_0^T\!\!\int_{\He}F'(\Phi_\eps)\eta\,\dx\dt\right|
 \le \|\eta\|_\infty |\He|^{1/2}
      \|F'(\Phi_\eps)\|_{L^1_tL^2_x}\to0.
\]
Together with the zero-filled gradient convergence, these estimates permit passage to the limit in \eqref{eq:weakchemical} for smooth tests. Continuity of both sides on \(L^2(0,T;H^1(\Omega))\) then gives \eqref{eq:limitchemical} by density.

The function $\eta_\eps=\phi_\eps-\phi|_{\Oe}$ belongs to $L^2(0,T;H^1(\Oe))$ and is therefore admissible in \eqref{eq:weakchemical}. Moreover, $\mu_\eps$ and $F'(\phi_\eps)$ are bounded in $L^2(0,T;L^2(\Oe))$ by \eqref{eq:muFbounds}, so every term in this test is at the energy level. Consequently,
\[
 \int_0^T\!\int_\Oe|\nabla\phi_\eps|^2\dx\dt
 =\int_0^T\!\int_\Oe\nabla\phi_\eps\cdot\nabla\phi\dx\dt
  +\int_0^T\!\int_\Oe
      (\mu_\eps-F'(\phi_\eps))(\phi_\eps-\phi)\dx\dt.
\]
The last term tends to zero by \eqref{eq:muFbounds} and strong $L^2$ convergence of the phase. The first term converges to $\|\nabla\phi\|_{L^2(\Dt)}^2$. Weak convergence and convergence of the norms give the first assertion of \eqref{eq:strongphase}.

The strong convergence on the fluid region must also be transferred to the extension inside the holes. By locality,
\[
 \int_0^T\!\int_{\He}|\nabla\Phi_\eps|^2\dx\dt
 \leq C\int_0^T\!\int_{\calA_\eps}|\nabla\phi_\eps|^2\dx\dt\to0.
\]
Since $\calA_\eps\subset\Oe$,
\[
 \begin{aligned}
 \int_0^T\!\int_{\calA_\eps}|\nabla\phi_\eps|^2\dx\dt
 &\leq 2\int_0^T\!\int_{\calA_\eps}
      |\chi_\eps\nabla\phi_\eps-\nabla\phi|^2\dx\dt
      +2\int_0^T\!\int_{\calA_\eps}|\nabla\phi|^2\dx\dt\\
 &\leq 2\|\chi_\eps\nabla\phi_\eps-\nabla\phi\|_{L^2(\Dt)}^2
      +2\int_0^T\!\int_{\calA_\eps}|\nabla\phi|^2\dx\dt
 \to0.
 \end{aligned}
\]
Here $\calA_\eps\subset\Oe$ and $|\calA_\eps|\to0$. The first term tends to zero by the strong convergence of the zero-filled gradients already established, whereas the second tends to zero by absolute continuity of the $L^1$ integral of $|\nabla\phi|^2$. The local extension estimate \eqref{eq:extensionlocal} then forces the gradient energy of the extension inside the holes to vanish. Together with \cref{lem:phaseL2}, this proves the full strong $L^2_tH^1_x$ convergence.

For the polynomial term, Sobolev embedding and the uniform $L^\infty_tL^6_x$ bound give
\[
 \|F'(\Phi_\eps)-F'(\phi)\|_{L^2_tL^2_x}
 \leq C\big(1+\|\Phi_\eps\|_{L^\infty_tL^6_x}^2
              +\|\phi\|_{L^\infty_tL^6_x}^2\big)
        \|\Phi_\eps-\phi\|_{L^2_tL^6_x}\to0.
\]
Finally, put $b_\eps=\nu(\Phi_\eps)$ and $b=\nu(\phi)$. The Lipschitz bound gives strong $L^2$ convergence. By the Sobolev chain rule,
\[
 \nabla b_\eps-\nabla b
 =\nu'(\Phi_\eps)(\nabla\Phi_\eps-\nabla\phi)
  +(\nu'(\Phi_\eps)-\nu'(\phi))\nabla\phi.
\]
The first term tends to zero in $L^2$. By \cref{lem:phaseL2}, after passage to the present subsequence $\Phi_\eps\to\phi$ almost everywhere in $\mathcal D_T$. Continuity and boundedness of $\nu'$ therefore give $\nu'(\Phi_\eps)\to\nu'(\phi)$ almost everywhere and boundedly. Dominated convergence applied to $|\nabla\phi|^2$ shows that the second term also tends to zero in $L^2$. This proves \eqref{eq:strongnonlinear}.
\end{proof}

\Cref{prop:strongphase} gives $\nu(\Eeps\phi_\eps)\to\nu(\phi)$ strongly in $L^2_tH^1_x$. The next section combines this convergence with the critical corrector estimates to identify the viscous capacity term.

\section{Critical solenoidal correctors and weighted capacity}\label{sec:correctors}
\subsection{A local divergence inverse}\label{subsec:localinverse}
The trace corrections below require a divergence right inverse whose zero extension retains high Sobolev regularity. Since the reference shell is a fixed smooth domain, we use the classical compact-support Bogovski\u\i{} construction on $S$.

For an open set $G$, set
\[
 C^\infty_{c,0}(G)=\left\{g\in C_c^\infty(G):\int_G g\,\mathrm dx=0\right\},
 \qquad S=B(0,1)\setminus\overline{B(0,1/2)}.
\]

\begin{lemma}[A divergence inverse on a fixed shell]\label{lem:localinverse}
There exists a linear operator
\[
 \mathfrak B_S:C^\infty_{c,0}(S)\to C_c^\infty(S;\mathbb R^3),
\]
such that the zero extension of $\mathfrak B_Sg$ to $\mathbb R^3$ is smooth and supported in $\overline S$ and
\begin{equation}\label{eq:referenceBogovskii}
 \Div(\mathfrak B_Sg)=g,
 \qquad
 \|\mathfrak B_Sg\|_{W^{s+1,p}(S)}
 \le C_{s,p}\|g\|_{W^{s,p}(S)},
 \quad s=0,1,2,\quad 1<p<\infty.
\end{equation}
For $x_0\in\mathbb R^3$ and $\rho>0$, let $A=x_0+\rho S$. If $g\in C^\infty_{c,0}(A)$, define
\begin{equation}\label{eq:scaledBogovskii}
 (\mathfrak B_Ag)(x_0+\rho y)
 =\rho\,\mathfrak B_S[g(x_0+\rho\,\cdot)](y),\qquad y\in S.
\end{equation}
Then the zero extension of $\mathfrak B_Ag$ is smooth and supported in $\overline A$, $\Div\mathfrak B_Ag=g$ and
\begin{equation}\label{eq:scaledBogovskiibound}
 \sum_{j=0}^{s+1}\rho^{j-1}
 \|\nabla^j\mathfrak B_Ag\|_{L^p(A)}
 \le C_{s,p}\sum_{j=0}^{s}\rho^j\|\nabla^jg\|_{L^p(A)}.
\end{equation}
The constants depend only on the fixed shell, $s$ and $p$.
\end{lemma}
\begin{proof}
The shell $S$ is a fixed bounded smooth domain. The compact-support Bogovski\u\i{} construction on such a domain gives a single linear operator \(\mathfrak B_S\) with the following property: if \(g\in C^\infty_{c,0}(S)\), then \(\mathfrak B_Sg\in C_c^\infty(S;\mathbb R^3)\), \(\Div(\mathfrak B_Sg)=g\) and for any integer \(m\ge0\) and \(1<p<\infty\),
\[
 \|\mathfrak B_Sg\|_{W^{m+1,p}(S)}
 \le C_{m,p,S}\|g\|_{W^{m,p}(S)}.
\]
A construction with these support and higher-order estimates is given in \cite[Chapter~III, Theorem~3.3 and Remark~3.12]{Galdi}, see also the compact-support mapping in \cite[Corollary~3.3]{CM} and the order-one Sobolev gain in \cite[Corollary~3.4 and Remark~3.5]{CM}. Taking \(m=s=0,1,2\) yields \eqref{eq:referenceBogovskii}. Because the output is compactly supported in \(S\), its zero extension is smooth on \(\mathbb R^3\) and supported in \(\overline S\).

For the rescaled shell, a change of variables gives
\[
 \|\nabla^j\mathfrak B_Ag\|_{L^p(A)}
 =\rho^{1-j+3/p}
 \|\nabla^j\mathfrak B_S[g(x_0+\rho\,\cdot)]\|_{L^p(S)}.
\]
Moreover,
\[
 \|\nabla^m[g(x_0+\rho\,\cdot)]\|_{L^p(S)}
 =\rho^{m-3/p}\|\nabla^m g\|_{L^p(A)}.
\]
Multiplying by $\rho^{j-1}$, summing in $j$, and using \eqref{eq:referenceBogovskii} gives \eqref{eq:scaledBogovskiibound}. The divergence identity follows directly from \eqref{eq:scaledBogovskii}.
\end{proof}

With the divergence repair fixed on the reference shell, the exterior Stokes fields can now be localized without changing incompressibility.

The exterior Stokes problem \eqref{eq:exterior} has a unique finite-energy velocity in the class $\vh^i\in L^6$, $\nabla \vh^i\in L^2$, see the classical exterior-domain theory in \cite{Galdi} and its use in \cite{AllaireI,HLO}. Smoothness of the reference obstacle and decay at infinity imply
\[
 |\vh^i(y)|\leq\frac C{1+|y|},\quad
 |\nabla \vh^i(y)|+|q^i(y)|\leq\frac C{(1+|y|)^2},\quad
 |\nabla^2\vh^i(y)|+|\nabla q^i(y)|\leq\frac C{(1+|y|)^3}.
\]
The bounds near $\partial K$ are interpreted on the exterior. Extend $\vh^i$ by $\ve_i$ inside $K$. Its distributional divergence is zero, and its gradient vanishes in $K$. Let $\vn_E$ be the outward normal of the exterior fluid, pointing into $K$. To identify the total traction, set \(\sigma^i=2D\vh^i-q^i\Ithree\). Testing \(-\Div\sigma^i=0\) by \(\vh^j\) on \(B_R\setminus K\), using \(\vh^j=\ve_j\) on \(\partial K\), and passing \(R\to\infty\) gives
\[
 \ve_j\cdot\int_{\partial K}\sigma^i\vn_E\ds
 =\int_{\R^3\setminus K}2D\vh^i:D\vh^j\dx=\M_{ij}.
\]
Since \(\M\) is symmetric by \eqref{eq:capacity}, this is equivalent to
\begin{equation}\label{eq:capacityidentities}
 \int_{\partial K}(2D\vh^i-q^i\Ithree)\vn_E\ds=\M\ve_i.
\end{equation}
If \(\vxi\cdot\M\vxi=0\), the exterior field \(\sum_i\vxi_i\vh^i\) has zero symmetric gradient and is therefore a rigid motion. Decay at infinity makes it vanish and the boundary trace then gives \(\vxi=0\). Thus \(\M\) is positive definite.

For the local construction, abbreviate $a=a_\eps$ and $\ell=\ell_\eps$, and put $R=\ell/4$. We use the operator of \cref{lem:localinverse}, extending its output by zero outside the shell.

Choose a radial $\vartheta\in C_c^\infty(B(0,7/8))$ that equals $1$ on $B(0,5/8)$ and set
\[
 \vartheta_k(x)=\vartheta((x-x_{\eps,k})/R),\quad
 S_{\eps,k}=x_{\eps,k}+RS.
\]
The scaled exterior fields are
\[
 \vh_{a,k}^i(x)=\vh^i((x-x_{\eps,k})/a),\quad
 q_{a,k}^i(x)=a^{-1}q^i((x-x_{\eps,k})/a).
\]
Since the constant extension of $\vh_{a,k}^i$ is solenoidal, the smooth datum $\nabla\vartheta_k\cdot \vh_{a,k}^i$ has zero integral and is compactly supported in $S_{\eps,k}$. Define
\[
 \vz_{\eps,k}^i=\mathfrak B_{S_{\eps,k}}
        (\nabla\vartheta_k\cdot \vh_{a,k}^i).
\]
By \eqref{eq:scaledBogovskiibound} with $s=2$ and interior derivative estimates for the exterior Stokes field on the shell,
\[
 \Div \vz_{\eps,k}^i=\nabla\vartheta_k\cdot \vh_{a,k}^i,
 \quad |\nabla^j\vz_{\eps,k}^i|\leq CaR^{-1-j},\quad j=0,1,2.
\]
The normalized datum $R(\nabla\vartheta_k\cdot \vh_{a,k}^i)(x_{\eps,k}+R\,\cdot)$ has $W^{2,p}(S)$ norm at most $Ca/R$. Taking $p>3$ and using $W^{3,p}(S)\hookrightarrow W^{2,\infty}(S)$ proves the pointwise bounds. Put
\begin{equation}\label{eq:localizedH}
 \vH_{\eps,k}^i=\vartheta_k\vh_{a,k}^i-\vz_{\eps,k}^i,
 \quad \varpi_{\eps,k}^i=\vartheta_kq_{a,k}^i.
\end{equation}
Extend $\vH_{\eps,k}^i$ by $\ve_i$ in $K_{\eps,k}$ and by zero outside $B(x_{\eps,k},R)$. It is globally solenoidal and belongs to $H^1(\Omega)$.

\begin{lemma}[Localized capacity fields]\label{lem:localized}
The fields in \eqref{eq:localizedH} satisfy
\begin{align}
 |\vH_{\eps,k}^i(x)|&\leq\frac{Ca}{a+|x-x_{\eps,k}|},&
 |\nabla \vH_{\eps,k}^i(x)|&\leq\frac{Ca}{(a+|x-x_{\eps,k}|)^2},\notag\\
 \int_\Oe2D\vH_{\eps,k}^i:D\vH_{\eps,k}^j\dx
 &=a\M_{ij}+O(a^2/R),&
 \|\vH_{\eps,k}^i\|_{L^2(\Omega)}^2&\leq Ca^2R.\label{eq:Henergy}
\end{align}
The pointwise estimates hold for almost every $x\in\Omega$ after the constant extension through $K_{\eps,k}$ and the zero extension outside $B(x_{\eps,k},R)$.
The distribution
\[
 \vG_{\eps,k}^i=-\Delta \vH_{\eps,k}^i+\nabla \varpi_{\eps,k}^i
\]
is a smooth shell-supported function as a distribution on the open fluid domain $\Oe$, and
\begin{equation}\label{eq:shellforcebounds}
 \|\vG_{\eps,k}^i\|_{L^\infty(\Oe)}\leq CaR^{-3},\quad
 \int_\Oe \vG_{\eps,k}^i\dx=-a\M \ve_i.
\end{equation}
For every $\vw\in \cV_\eps$,
\begin{equation}\label{eq:shellpairing}
 \int_\Oe2D\vw:D\vH_{\eps,k}^i\dx
       =\int_\Oe \vG_{\eps,k}^i\cdot \vw\dx.
\end{equation}
\end{lemma}
\begin{proof}
The exterior decay and the pointwise bounds for the shell repair give the first two estimates in \eqref{eq:Henergy}, interpreted almost everywhere for the constant extension through the obstacle. For $|x-x_{\eps,k}|\lesssim a$, they reduce to the fixed-scale bounds $O(1)$ and $O(a^{-1})$. For $|x-x_{\eps,k}|\gtrsim a$, they are exactly the exterior decay rates $O(a/r)$ and $O(a/r^2)$. We record the energy comparison because it fixes the critical normalization. On the annulus \(|x-x_{\eps,k}|\ge R/2\),
\[
 |\nabla\vh_{a,k}^i(x)|\le C\frac{a}{|x-x_{\eps,k}|^2},
\]
hence
\[
 \int_{\{|x-x_{\eps,k}|>R/2\}}|\nabla\vh_{a,k}^i|^2\dx
 \le C a^2\int_{R/2}^{\infty}r^{-2}\,\mathrm dr
 \le C\frac{a^2}{R}.
\]
On the cutoff shell,
\[
 |\nabla\vartheta_k|\,|\vh_{a,k}^i|
 \le C R^{-1}\frac aR,
 \qquad
 |\nabla\vz_{\eps,k}^i|\le C\frac a{R^2},
\]
and therefore
\[
 \int_{S_{\eps,k}}
 \big(|\nabla\vartheta_k|^2|\vh_{a,k}^i|^2
      +|\nabla\vz_{\eps,k}^i|^2\big)\dx
 \le C\frac{a^2}{R}.
\]
The same estimates hold for the corresponding cross terms by Cauchy--Schwarz. With the change of variables $y=(x-x_{\eps,k})/a$, one has $D_x\vh_{a,k}^i=a^{-1}D_y\vh^i$ and $\mathrm dx=a^3\,\mathrm dy$. Hence the full rescaled exterior energy is
\[
 \int_{\R^3\setminus K_{\eps,k}}
 2D\vh_{a,k}^i:D\vh_{a,k}^j\dx=a\M_{ij}.
\]
Removing the exterior tail and inserting the cutoff and divergence repair changes the energy by at most \(C a^2/R\). This proves the energy identity in \eqref{eq:Henergy}. The \(L^2\) estimate follows from
\[
 \int_0^R\left(\frac{a}{a+r}\right)^2r^2\,\mathrm dr
 \le a^2R,
\]
together with the $O(a^3)$ contribution of the constant extension inside the obstacle. Since $a/R\to0$, this contribution is bounded by $Ca^2R$ for sufficiently small $\eps$.

The localized force is supported in the cutoff shell. Differentiating
\(\vartheta_k\vh_{a,k}^i-\vz_{\eps,k}^i\) and
\(\vartheta_k q_{a,k}^i\), using
\[
 |\nabla^j\vartheta_k|\le C R^{-j},\qquad
 |\vh_{a,k}^i|\le C\frac aR,\qquad
 |\nabla\vh_{a,k}^i|+|q_{a,k}^i|\le C\frac a{R^2},
\]
and
\(|\nabla^2\vz_{\eps,k}^i|\le CaR^{-3}\),
gives
\(\|\vG_{\eps,k}^i\|_\infty\le CaR^{-3}\).
Because the localized fields agree with the rescaled exterior Stokes solution near \(\partial K_{\eps,k}\), while $\vartheta_k$, $\varpi_{\eps,k}^i$, and the compactly supported repair $\vz_{\eps,k}^i$ all vanish in a neighborhood of $\partial B(x_{\eps,k},R)$, both the localized velocity and its stress vanish there. Therefore
\[
 \int_\Oe\vG_{\eps,k}^i\dx
 =-\int_{\partial K_{\eps,k}}
   (2D\vh_{a,k}^i-q_{a,k}^i\Ithree)\vn_E\,\ds.
\]
The change of variables \(x=x_{\eps,k}+ay\) contributes one power of \(a\) to the traction integral and \eqref{eq:capacityidentities} therefore yields
\[
 \int_\Oe\vG_{\eps,k}^i\dx=-a\M\ve_i.
\]
Finally, integration by parts against \(\vw\in\cV_\eps\) gives
\[
 \int_\Oe2D\vw:D\vH_{\eps,k}^i\dx
 =\int_\Oe\vG_{\eps,k}^i\cdot\vw\dx,
\]
because \(\Div\vw=0\) and \(\vw\) has zero trace on every solid boundary. Density extends the identity to all of \(\cV_\eps\).
\end{proof}

To use the local defects as oscillating tests, we next match the full trace of a smooth solenoidal field on each obstacle.

Let $\vpsi$ be smooth and solenoidal, compactly supported if $\Omega\subset \mathbb{R}^3$, and periodic if $\Omega=\T$.
After correcting the cell-center values, a second local field removes the smooth variation of the trace across each hole.

For each $k$, set $d_k[\vpsi](x)=\vpsi(x)-\vpsi(x_{\eps,k})$. Choose $\eta\in C_c^\infty(B(0,7/16))$ equal to $1$ on $B(0,5/16)$, and put $\eta_k(x)=\eta((x-x_{\eps,k})/a)$. The shell $S_{\eps,k}^{a}=x_{\eps,k}+(a/2)S$ lies outside $K_{\eps,k}$. Using the operator in \eqref{eq:scaledBogovskii}, set
\[
 \vZ_{\eps,k}[\vpsi]=\eta_k d_k[\vpsi]
       -\mathfrak B_{S_{\eps,k}^{a}}(\nabla\eta_k\cdot d_k[\vpsi]).
\]
The datum has zero integral because $\Div d_k[\vpsi]=0$ and $\eta_kd_k[\vpsi]$ is compactly supported in the full ball. Thus $\vZ_{\eps,k}[\vpsi]$ is solenoidal, equals $\vpsi-\vpsi(x_{\eps,k})$ near $K_{\eps,k}$, and vanishes outside $B(x_{\eps,k},a/2)$. Scaling and smoothness of $\vpsi$ give
\begin{equation}\label{eq:Zbounds}
 \|\vZ_{\eps,k}[\vpsi]\|_\infty\leq C_{\vpsi} a,
 \quad \|\nabla \vZ_{\eps,k}[\vpsi]\|_\infty\leq C_{\vpsi},
 \quad
 \Big\|\sum_k \vZ_{\eps,k}[\vpsi]\Big\|_{H^1(\Omega)}\leq C_{\vpsi}\ell^3.
\end{equation}
The construction is linear in $\vpsi$. On $B(x_{\eps,k},a/2)$, Taylor's formula gives $|d_k[\vpsi]|\leq C_{\vpsi}a$ and $|\nabla d_k[\vpsi]|\leq C_{\vpsi}$. The scaled bound \eqref{eq:scaledBogovskiibound} gives the same orders for the divergence repair. Since the supports are disjoint and $\#\Ie\,a^3\leq C\ell^6$,
\[
 \Big\|\sum_k\vZ_{\eps,k}[\vpsi]\Big\|_{L^2(\Omega)}\leq C_{\vpsi}a\ell^3,
 \quad
 \Big\|\sum_k\nabla\vZ_{\eps,k}[\vpsi]\Big\|_{L^2(\Omega)}\leq C_{\vpsi}\ell^3,
\]
which yields the last estimate in \eqref{eq:Zbounds}.

With the local fields now fixed, define the linear correction $\Ceps:\Ssig\to \cV$ by
\begin{equation}\label{eq:correctordef}
 \Ceps\vpsi
 =\vpsi-\sum_{k\in\Ie}\sum_{i=1}^3
       \vpsi_i(x_{\eps,k})\vH_{\eps,k}^i
       -\sum_{k\in\Ie}\vZ_{\eps,k}[\vpsi].
\end{equation}
This formula defines the ambient, zero-extended field. Its restriction to $\Oe$ is the microscopic test. It vanishes on every hole and is solenoidal. All corrections lie strictly inside their cells, so the outer trace of $\vpsi$ is unchanged. The cutoffs and right inverses are independent of time.

\begin{lemma}[Corrected test bounds]\label{lem:corrector}
The restriction of $\Ceps\vpsi$ belongs to $\cV_\eps$ and
\[
 \|\Ceps\vpsi\|_\infty+\|\Ceps\vpsi\|_{H^1(\Omega)}\leq C_{\vpsi},
 \quad \|\Ceps\vpsi-\vpsi\|_2\leq C_{\vpsi}\ell_\eps^2
\]
hold. Moreover, $\Ceps\vpsi$ converges weakly to $\vpsi$ in $H^1(\Omega)$ and strongly in $L^q(\Omega)$ for any $1\leq q<\infty$. For smooth time-dependent tests these bounds hold uniformly in time, and
\begin{equation}\label{eq:correctortime}
 \partial_t\Ceps\vpsi=\Ceps(\partial_t\vpsi).
\end{equation}
Moreover, for almost every $x\in\Oe$,
\begin{equation}\label{eq:correctorpointwise}
 |D\Ceps\vpsi(x)|\leq C_{\vpsi}\left(1+
 \sum_{k\in\Ie}\frac{a_\eps\1_{B(x_{\eps,k},\ell_\eps/4)}(x)}
 {(a_\eps+|x-x_{\eps,k}|)^2}\right).
\end{equation}
\end{lemma}
\begin{proof}
Solenoidality and all boundary traces follow directly from the construction. Disjointness, \eqref{eq:Henergy} and \eqref{eq:Zbounds} give
\[
 \sum_k\|\vH_{\eps,k}^i\|_2^2\leq C\ell^{-3}a^2R\leq C\ell^4,
 \quad
 \sum_k\|\nabla \vH_{\eps,k}^i\|_2^2\leq C\ell^{-3}a\leq C.
\]
The gradient estimates and \eqref{eq:Zbounds} give the uniform $H^1$ bound. For the $L^2$ error, disjointness and boundedness of the cell-center values yield
\[
 \Big\|\sum_{k,i}\vpsi_i(x_{\eps,k})\vH_{\eps,k}^i\Big\|_2^2
 \leq C_{\vpsi}\#\Ie\,a^2R\leq C_{\vpsi}\ell^4,
 \quad
 \Big\|\sum_k\vZ_{\eps,k}[\vpsi]\Big\|_2^2
 \leq C_{\vpsi}a^2(\#\Ie\,a^3)\leq C_{\vpsi}\ell^{12}.
\]
Hence $\|\Ceps\vpsi-\vpsi\|_2\leq C_{\vpsi}\ell^2$. The pointwise bounds imply the $L^\infty$ estimate and \eqref{eq:correctorpointwise}. Interpolating the $L^2$ error with the uniform $L^\infty$ bound gives the explicit rate used below,
\[
 \|\Ceps\vpsi-\vpsi\|_{L^3(\Omega)}
 \leq \|\Ceps\vpsi-\vpsi\|_2^{2/3}
          \|\Ceps\vpsi-\vpsi\|_\infty^{1/3}
 \leq C_{\vpsi}\ell_\eps^{4/3}.
\]
In particular, the difference converges strongly in every finite $L^q(\Omega)$ by interpolation with the uniform $L^\infty$ bound. Distributional differentiation and the uniform $H^1$ bound identify the weak gradient limit. The geometry and all local repairs are time-independent and linear, so \eqref{eq:correctortime} follows.
\end{proof}

\subsection{A capacity-scale multiplier estimate}
At the critical scale the corrector gradient carries order-one $L^2$ energy. Its product with the viscosity difference is controlled by the following multiplier estimate.

\begin{lemma}[Hardy multiplier at critical scale]\label{lem:multiplier}
For any scalar $h\in H^1(\Omega)$ and any test $\vpsi$ as in \cref{lem:corrector},
\begin{equation}\label{eq:multiplier}
 \|hD\Ceps\vpsi\|_{L^2(\Oe)}\leq C_{\vpsi}\|h\|_{H^1(\Omega)}.
\end{equation}
The constant is independent of $\eps$. In particular, for time-dependent $h$ and smooth $\vpsi$,
\begin{equation}\label{eq:timemultiplier}
 \|hD\Ceps\vpsi\|_{L^2((0,T)\times\Oe)}
       \leq C_{\vpsi}\|h\|_{L^2_tH^1_x}.
\end{equation}
\end{lemma}
\begin{proof}
Fix a cell $Q=Q_{\eps,k}$, write $x_k=x_{\eps,k}$, and put $\bar h_Q=\fint_Qh$. To obtain the localized multiplier bound, choose $\chi_Q\in C_c^\infty(Q)$ with $\chi_Q=1$ on $B(x_k,\ell/4)$ and $|\nabla\chi_Q|\le C/\ell$. Applying Hardy's inequality in $\R^3$ to $\chi_Q(h-\bar h_Q)$ yields
\[
 \int_{B(x_k,\ell/4)}\frac{|h-\bar h_Q|^2}{|x-x_k|^2}\,\mathrm dx
 \le C\int_Q|\nabla h|^2\,\mathrm dx
   +C\ell^{-2}\int_Q|h-\bar h_Q|^2\,\mathrm dx.
\]
The second term is bounded by $C\int_Q|\nabla h|^2\,\mathrm dx$ by the cell Poincar\'e inequality. Since $a^2/(a+r)^4\le r^{-2}$, we obtain
\[
 \int_{B(x_k,\ell/4)}\frac{a^2}{(a+|x-x_k|)^4}|h-\bar h_Q|^2\dx
 \leq C\int_Q|\nabla h|^2\dx.
\]
For the mean part,
\[
 \int_{B(x_k,\ell/4)}\frac{a^2}{(a+|x-x_k|)^4}\dx\leq Ca,
 \quad
 a|\bar h_Q|^2\leq\frac a{\ell^3}\int_Q|h|^2\dx.
\]
Consequently,
\[
 \int_{B(x_k,\ell/4)}\frac{a^2|h|^2}{(a+|x-x_k|)^4}\dx
 \leq C\int_Q|\nabla h|^2\dx
       +C\frac a{\ell^3}\int_Q|h|^2\dx.
\]
Summing over the disjoint cells, using $a/\ell^3\le C$--the uniform bound supplied by the critical assumption $a_\eps/\ell_\eps^3\to\beta$--and \eqref{eq:correctorpointwise}, proves \eqref{eq:multiplier}. Integration in time proves \eqref{eq:timemultiplier}. On the torus the same argument is performed in periodic cell coordinates.
\end{proof}

The multiplier estimate requires a bounded capacity density $a_\eps/\ell_\eps^3$. This condition also includes the small-hole regime in which that ratio tends to zero. The critical feature here is its positive limit $\beta$, which leaves a nonzero corrector energy and produces the Brinkman resistance. The estimate applies at Sobolev energy regularity for the microscopic coefficient.

\subsection{Mixed limits and the viscous lower bound}
The weighted argument begins with a smooth coefficient satisfying fixed positive lower and upper bounds. The Hardy multiplier estimate then extends the pairing to the Sobolev coefficients of \cref{prop:weighted}.

\begin{lemma}[Smooth-coefficient capacity pairing]\label{lem:smoothcapacity}
Let $b$ be smooth on $[0,T]\times\overline\Omega$ and let $\vpsi$ be a smooth solenoidal test, with common compact spatial support in a domain or periodic on the torus. For $\vw_\eps\in L^2(0,T;\cV_\eps)$, define
\[
 \begin{aligned}
 \mathcal R_\eps[b,\vw_\eps,\vpsi]
 =\int_0^T\!\int_\Oe2bD\vw_\eps:D\Ceps\vpsi\dx\dt
 -\int_{\Dt}\big(2bD\widetilde \vw_\eps:D\vpsi
             +b\B\widetilde \vw_\eps\cdot\vpsi\big)\dx\dt.
 \end{aligned}
\]
For sufficiently small $\eps$,
\begin{equation}\label{eq:smoothremainderbound}
 \begin{aligned}
 |\mathcal R_\eps[b,\vw_\eps,\vpsi]|
 \leq C_{\vpsi,T}
       \big(\|b\|_{L^\infty(\Dt)}+\|\nabla b\|_{L^\infty(\Dt)}\big)
\times\left(\ell_\eps+
          \left|\frac{a_\eps}{\ell_\eps^3}-\beta\right|\right)
       \|\widetilde \vw_\eps\|_{L^2_tH^1_x}.
 \end{aligned}
\end{equation}
Consequently, if $\widetilde \vw_\eps\weak \vw$ in $L^2(0,T;\cV)$, then
\begin{equation}\label{eq:smoothmixed}
 \int_0^T\!\int_\Oe2bD\vw_\eps:D\Ceps\vpsi\dx\dt
 \to\int_{\Dt}2bD\vw:D\vpsi\dx\dt
               +\int_{\Dt}b\B \vw\cdot\vpsi\dx\dt.
\end{equation}
For the recovery field itself,
\begin{equation}\label{eq:smoothrecovery}
 \int_0^T\!\int_\Oe2b|D\Ceps\vpsi|^2\dx\dt
 \to\int_{\Dt}\big(2b|D\vpsi|^2+b\B\vpsi\cdot\vpsi\big)\dx\dt.
\end{equation}
\end{lemma}
\begin{proof}
Begin with $b=1$ at a fixed time. Set
$\bar \vw_{\eps,k}=\ell^{-3}\int_{Q_{\eps,k}}\widetilde \vw_\eps\dx$. The shell force and the cell Poincar\'e inequality imply
\begin{align*}
 \left|\int_\Oe \vG_{\eps,k}^i\cdot
                   (\vw_\eps-\bar \vw_{\eps,k})\dx\right|
 &\leq C aR^{-3/2}\ell
       \|\nabla\widetilde \vw_\eps\|_{L^2(Q_{\eps,k})}\\
 &\leq Ca\ell^{-1/2}
       \|\nabla\widetilde \vw_\eps\|_{L^2(Q_{\eps,k})}.
\end{align*}
Summation and Cauchy--Schwarz give a total error bounded by
\[
 C_{\vpsi} a\ell^{-2}\|\nabla\widetilde \vw_\eps\|_2
       \leq C_{\vpsi}\ell\|\nabla\widetilde \vw_\eps\|_2.
\]
Using \eqref{eq:shellforcebounds}--\eqref{eq:shellpairing}, we obtain the explicit discrete pairing estimate
\begin{equation}\label{eq:discretecapacity}
 \begin{aligned}
 \left|\sum_{k\in\Ie}\sum_{i=1}^3\vpsi_i(x_{\eps,k})
       \int_\Oe2D\vw_\eps:D\vH_{\eps,k}^i\dx
       +a\sum_{k\in\Ie}\M\vpsi(x_{\eps,k})\cdot\bar \vw_{\eps,k}\right|
 \leq C_{\vpsi} a\ell^{-2}\|\nabla\widetilde \vw_\eps\|_2.
 \end{aligned}
\end{equation}
The sum of the cell averages is an integral against a piecewise constant coefficient. For smooth $b$, its comparison with the limiting coefficient is quantified by
\begin{equation}\label{eq:cellcoefficienterror}
 \begin{aligned}
 \left\|\frac a{\ell^3}\sum_k
       b(t,x_{\eps,k})\M\vpsi(t,x_{\eps,k})\1_{Q_{\eps,k}}
       -\beta b(t,\cdot)\M\vpsi(t,\cdot)\right\|_2\\
 \leq C_{\vpsi}\big(\|b\|_{L^\infty(\Dt)}+
                     \|\nabla b\|_{L^\infty(\Dt)}\big)
           \left(\ell+\left|\frac a{\ell^3}-\beta\right|\right),
 \end{aligned}
\end{equation}
uniformly in time. This follows from the mean-value estimate on each cell and the bounded ratio $a/\ell^3$. In a domain, the common support of $\vpsi$ is disjoint from the omitted boundary strip for sufficiently small $\eps$. On the torus all cells are present up to their faces.

Introduce, for fixed $t$,
\[
 b_{\eps,k}=b(t,x_{\eps,k}),\qquad
 \vpsi_{\eps,k}=\vpsi(t,x_{\eps,k}),\qquad
 \bar\vw_{\eps,k}=\ell^{-3}\int_{Q_{\eps,k}}\widetilde\vw_\eps\,\mathrm dx,
\]
and the piecewise constant coefficient
\[
 \boldsymbol q_\eps(t,x)=\frac a{\ell^3}\sum_{k\in\Ie}
 b_{\eps,k}\M\vpsi_{\eps,k}\,\1_{Q_{\eps,k}}(x).
\]
Because the zero extension of $\vw_\eps$ has weak gradient equal to the zero extension of $\nabla\vw_\eps$, the bulk term in the definition of $\mathcal R_\eps$ cancels exactly. Using \eqref{eq:correctordef}, the remainder admits the exact decomposition
\begin{equation}\label{eq:remainderdecomposition}
 \mathcal R_\eps
 =\mathcal R_\eps^{\mathrm{shell}}
 +\mathcal R_\eps^{\mathrm{freeze}}
 +\mathcal R_\eps^{\mathrm{cell}}
 +\mathcal R_\eps^{Z},
\end{equation}
where
\begin{align*}
 \mathcal R_\eps^{\mathrm{shell}}
 &:=-\int_0^T\sum_{k\in\Ie} b_{\eps,k}
 \left[
 \sum_{i=1}^3\vpsi_{\eps,k,i}
 \int_{\Oe}2D\vw_\eps:D\vH_{\eps,k}^i\,\mathrm dx
 +a\M\vpsi_{\eps,k}\cdot\bar\vw_{\eps,k}
 \right]\,\mathrm dt,\\
 \mathcal R_\eps^{\mathrm{freeze}}
 &:=-\int_0^T\sum_{k\in\Ie}\sum_{i=1}^3
 \vpsi_{\eps,k,i}
 \int_{\Oe}2\bigl(b-b_{\eps,k}\bigr)
 D\vw_\eps:D\vH_{\eps,k}^i\,\mathrm dx\,\mathrm dt,\\
 \mathcal R_\eps^{\mathrm{cell}}
 &:=\int_{\Dt}\bigl(\boldsymbol q_\eps-\beta b\M\vpsi\bigr)
 \cdot\widetilde\vw_\eps\,\mathrm dx\,\mathrm dt,\\
 \mathcal R_\eps^{Z}
 &:=-\int_0^T\!\int_{\Oe}
 2bD\vw_\eps:D\!\left(\sum_{k\in\Ie}\vZ_{\eps,k}[\vpsi]\right)
 \,\mathrm dx\,\mathrm dt.
\end{align*}
Indeed,
\[
 \int_\Omega \boldsymbol q_\eps\cdot\widetilde\vw_\eps\,\mathrm dx
 =a\sum_{k\in\Ie}b_{\eps,k}\M\vpsi_{\eps,k}\cdot\bar\vw_{\eps,k},
\]
so the two inserted cell terms in \eqref{eq:remainderdecomposition} cancel algebraically, while $\B=\beta\M$ supplies the target Brinkman term.

The four pieces can now be estimated independently. Applying \eqref{eq:discretecapacity} at each time with coefficients $b_{\eps,k}\vpsi_{\eps,k}$ and then Cauchy--Schwarz in time gives
\[
 |\mathcal R_\eps^{\mathrm{shell}}|
 \le C_{\vpsi,T}\|b\|_{L^\infty(\Dt)}\,
 \ell\,\|\widetilde\vw_\eps\|_{L^2_tH^1_x}.
\]
On the support of $\vH_{\eps,k}^i$, $|b-b_{\eps,k}|\le R\|\nabla b\|_\infty\le C\ell\|\nabla b\|_\infty$. Hence
\[
 |\mathcal R_\eps^{\mathrm{freeze}}|
 \le C_{\vpsi}\ell\|\nabla b\|_{L^\infty(\Dt)}
 \int_0^T\|D\widetilde\vw_\eps(t)\|_2
 \left(\sum_{k,i}\|D\vH_{\eps,k}^i\|_2^2\right)^{1/2}\!\mathrm dt.
\]
By \eqref{eq:Henergy},
\[
 \sum_{k,i}\|D\vH_{\eps,k}^i\|_2^2
 \le C\#\Ie\left(a+\frac{a^2}{R}\right)
 \le C,
\]
because $\#\Ie\le C\ell^{-3}$, $a/\ell^3$ is bounded and $R=\ell/4$. Therefore
\[
 |\mathcal R_\eps^{\mathrm{freeze}}|
 \le C_{\vpsi,T}\ell\|\nabla b\|_{L^\infty(\Dt)}
 \|\widetilde\vw_\eps\|_{L^2_tH^1_x}.
\]
The coefficient estimate \eqref{eq:cellcoefficienterror} gives directly
\[
 |\mathcal R_\eps^{\mathrm{cell}}|
 \le C_{\vpsi,T}
 \bigl(\|b\|_\infty+\|\nabla b\|_\infty\bigr)
 \left(\ell+\left|\frac a{\ell^3}-\beta\right|\right)
 \|\widetilde\vw_\eps\|_{L^2_tL^2_x}.
\]
Equivalently, if
$\Pi_\eps\widetilde\vw_\eps=\sum_k\bar\vw_{\eps,k}\1_{Q_{\eps,k}}$, then Jensen's inequality yields
\[
 \|\Pi_\eps\widetilde\vw_\eps\|_{L^2(\Omega)}^2
 =\sum_k\ell^3|\bar\vw_{\eps,k}|^2
 \le\|\widetilde\vw_\eps\|_{L^2(\Omega)}^2,
\]
which gives the same stability in the cell-average representation. Finally, \eqref{eq:Zbounds} implies
\[
 |\mathcal R_\eps^Z|
 \le C_{\vpsi,T}\|b\|_\infty\ell^3
 \|\widetilde\vw_\eps\|_{L^2_tH^1_x}.
\]
Combining these four bounds with \eqref{eq:remainderdecomposition} proves \eqref{eq:smoothremainderbound}. The sign of the Brinkman term is fixed by the negative total shell force in \eqref{eq:shellforcebounds} together with the subtraction of the defect field in \eqref{eq:correctordef}.

If $\widetilde\vw_\eps\rightharpoonup\vw$ in $L^2(0,T;\cV)$, estimate \eqref{eq:smoothremainderbound} tends to zero because $\ell_\eps\to0$ and $a_\eps/\ell_\eps^3\to\beta$. Weak convergence of the two ambient terms then gives \eqref{eq:smoothmixed}.

To obtain \eqref{eq:smoothrecovery}, apply \eqref{eq:smoothmixed} with $\vw_\eps=\Ceps\vpsi$. By \cref{lem:corrector}, its zero extension converges weakly to $\vpsi$ in $L^2(0,T;\cV)$. Hence the mixed limit gives exactly
\[
 \lim_{\eps\to0}\int_0^T\!\int_\Oe2b|D\Ceps\vpsi|^2\dx\dt
       =\int_{\Dt}\big(2b|D\vpsi|^2+b\B\vpsi\cdot\vpsi\big)\dx\dt.
\]

\end{proof}

The smooth pairing depends only on the perforated velocity spaces. For variable coefficients, \cref{lem:multiplier} controls the replacement error inside the concentrated gradient layer.

\begin{proposition}[Weighted critical limit]\label{prop:weighted}
Let $0<b_*\le b^*<\infty$. Suppose that $b_\eps,b\in L^2(0,T;H^1(\Omega))$ satisfy
\begin{equation}\label{eq:coeffhyp}
 0<b_*\leq b_\eps,b\leq b^*\quad\text{for almost every }(t,x)\in\Dt,\quad
 b_\eps\to b\quad\text{strongly in }L^2_tH^1_x.
\end{equation}
Let $\vw_\eps\in L^2(0,T;\cV_\eps)$ and assume $\widetilde \vw_\eps\weak \vw$ in $L^2(0,T;\cV)$. For any smooth time-dependent test $\vpsi$ with $\vpsi(t)\in\Ssig$,
\begin{equation}\label{eq:weightedmixed}
 \int_0^T\!\int_\Oe2b_\eps D\vw_\eps:D\Ceps\vpsi\dx\dt
 \to\int_{\Dt}2bD\vw:D\vpsi\dx\dt
                 +\int_{\Dt}b\B \vw\cdot\vpsi\dx\dt.
\end{equation}
The recovery identity \eqref{eq:smoothrecovery} remains valid with $b_\eps$ on the left and $b$ on the right. In addition,
\begin{equation}\label{eq:weightedliminf}
 \liminf_{\eps\to0}\int_0^T\!\int_\Oe2b_\eps|D\vw_\eps|^2\dx\dt
 \geq\int_{\Dt}\big(2b|D\vw|^2+b\B \vw\cdot \vw\big)\dx\dt.
\end{equation}
The same assertions hold on every subinterval of $(0,T)$.
\end{proposition}
\begin{proof}
By \cref{lem:multiplier}, we have
\begin{equation}\label{eq:coefficientremainder}
 \left|\int_0^T\!\int_\Oe2(b_\eps-b)D\vw_\eps:D\Ceps\vpsi\dx\dt\right|
 \leq C_{\vpsi}\|D\widetilde \vw_\eps\|_{L^2(\Dt)}
                   \|b_\eps-b\|_{L^2_tH^1_x}\to0.
\end{equation}
Assume first that $\Omega$ is bounded. Choose a bounded Lipschitz domain $\Omega'$ such that $\overline\Omega\Subset\Omega'$. Since $\Omega$ is Lipschitz, there is a bounded linear extension operator $\mathscr E_\Omega:H^1(\Omega)\to H^1(\mathbb R^3)$. After reflection in time across $t=0$ and $t=T$, restrict the extension to $(-T,2T)\times\Omega'$ and set
\[
 \widehat b=\mathcal T(\mathscr E_\Omega b),
 \qquad
 \mathcal T(r)=\min\{b^*,\max\{b_*,r\}\}.
\]
Because $\mathcal T$ is $1$-Lipschitz, the Sobolev chain rule gives
\[
 |\nabla\widehat b|\le |\nabla\mathscr E_\Omega b|
 \quad\text{a.e. on }(-T,2T)\times\Omega'.
\]
Moreover, $b_*\le \widehat b\le b^*$ and $\widehat b=b$ on $(0,T)\times\Omega$. Let $\rho_\delta$ be a nonnegative standard space--time mollifier with unit mass and choose
\[
 0<\delta<\tfrac12\min\{T,\operatorname{dist}(\Omega,\partial\Omega')\}.
\]
Then
\[
 b_\delta=(\rho_\delta*\widehat b)|_{(0,T)\times\Omega}
\]
is smooth on $[0,T]\times\overline\Omega$, satisfies
\[
 b_*\le b_\delta\le b^*,
 \qquad
 b_\delta\to b
 \quad\text{strongly in }L^2(0,T;H^1(\Omega)),
\]
and can therefore be used in \cref{lem:smoothcapacity}. On the torus, use the periodic spatial extension and the same reflected-time mollification. The multiplier estimate gives
\begin{equation}\label{eq:smoothinguniform}
 \begin{aligned}
 &\sup_\eps\left|\int_0^T\!\int_\Oe
        2(b-b_\delta)D\vw_\eps:D\Ceps\vpsi\dx\dt\right|\\
 &\quad\leq C_{\vpsi}
       \left(\sup_\eps\|D\widetilde \vw_\eps\|_{L^2(\Dt)}\right)
       \|b-b_\delta\|_{L^2_tH^1_x}.
 \end{aligned}
\end{equation}
Fix $\delta>0$ and apply \cref{lem:smoothcapacity} to $b_\delta$. The smooth-coefficient remainder may depend on this fixed approximation. After passing $\eps\to0$, the uniform estimate \eqref{eq:smoothinguniform} permits the limit $\delta\to0$. In the limiting integrals the replacement error is bounded by
\[
 C_{\vpsi}\|\vw\|_{L^2_tH^1_x}\|b-b_\delta\|_{L^2_tL^2_x}\to0,
\]
since $D\vpsi$, $\vpsi$ and $\B$ are bounded. Together with \eqref{eq:coefficientremainder}--\eqref{eq:smoothinguniform}, this proves \eqref{eq:weightedmixed} in the stated order of limits.

For the recovery identity, \cref{lem:corrector} gives $\Ceps\vpsi\in L^2(0,T;\cV_\eps)$ and a uniform $L^2(0,T;H^1(\Omega))$ bound for its zero extension. Moreover, $\|\Ceps\vpsi-\vpsi\|_{L^2(\Dt)}\to0$, so every weak $L^2(0,T;\cV)$ limit of the corrected tests equals $\vpsi$. Hence $\widetilde{\Ceps\vpsi}\rightharpoonup\vpsi$ in $L^2(0,T;\cV)$ and \eqref{eq:weightedmixed} applies with $\vw_\eps=\Ceps\vpsi$. This proves the weighted recovery identity with the same coefficient approximation and capacity normalization.

For $\vz,\vv\in L^2(0,T;\cV_\eps)$ and, respectively, $\vz,\vv\in L^2(0,T;\cV)$, define
\[
 \mathfrak a_\eps(\vz,\vv)=\int_0^T\!\int_\Oe2b_\eps D\vz:D\vv\dx\dt,
 \quad
 \mathfrak a_b(\vz,\vv)=\int_{\Dt}\big(2bD\vz:D\vv+b\B \vz\cdot \vv\big)\dx\dt.
\]
For a smooth solenoidal $\vv$, positivity gives
\[
 \mathfrak a_\eps(\vw_\eps,\vw_\eps)
 \geq2\mathfrak a_\eps(\vw_\eps,\Ceps \vv)-\mathfrak a_\eps(\Ceps \vv,\Ceps \vv).
\]
Taking the lower limit and using the two identities already proved yields
\[
 \liminf \mathfrak a_\eps(\vw_\eps,\vw_\eps)\geq2\mathfrak a_b(\vw,\vv)-\mathfrak a_b(\vv,\vv).
\]
Smooth solenoidal fields of the prescribed type are dense in \(L^2(0,T;\cV)\). The form \(\mathfrak a_b\) is continuous and coercive. In a bounded domain, Korn--Poincar\'e gives
\[
 \mathfrak a_b(\vz,\vz)\ge 2b_*\|D\vz\|_{L^2(\Dt)}^2
 \ge c\|\vz\|_{L^2(0,T;H^1_0(\Omega))}^2.
\]
On the torus, let $\lambda_{\min}(\B)>0$ denote the smallest eigenvalue of the positive-definite matrix $\B$. Then
\[
 \mathfrak a_b(\vz,\vz)
 \ge 2b_*\|D\vz\|_{L^2(\Dt)}^2
    +b_*\lambda_{\min}(\B)\|\vz\|_{L^2(\Dt)}^2
 \ge c\|\vz\|_{L^2(0,T;H^1(\Omega))}^2,
\]
where the last inequality is the periodic Korn inequality with the \(L^2\) term retaining the constant modes. Approximate \(\vw\) by smooth solenoidal \(\vv\) in this norm to obtain \eqref{eq:weightedliminf}. Restricting all integrals and coefficient convergences to a subinterval proves the final assertion.
\end{proof}

The recovery identity concerns prescribed corrected smooth tests. For a general weakly convergent velocity sequence, the variational conclusion at this level is the lower bound \eqref{eq:weightedliminf}.

\begin{remark}\label{rem:coefficientexample}
The $H^1$ topology in \cref{prop:weighted} reflects the capacity scale of the corrector gradients. On the torus, fix $i\in\{1,2,3\}$ and choose $0\leq\eta_0\in C_c^\infty(U)$ such that
\[
 \int_U\eta_0|D\vh^i|^2\,\mathrm dx>0.
\]
For $c>0$, set
\[
 b_\eps(x)=1+c\sum_k\eta_0\!\left(\frac{x-x_{\eps,k}}{a_\eps}\right).
\]
The supports are disjoint with $\#\Ie=\ell_\eps^{-3}$ and $a_\eps/\ell_\eps^3\to\beta$.  Hence
\[
 \|b_\eps-1\|_2^2
 =c^2\frac{a_\eps^3}{\ell_\eps^3}\|\eta_0\|_{L^2(U)}^2\to0,
 \quad
 \|\nabla b_\eps\|_2^2
 =c^2\frac{a_\eps}{\ell_\eps^3}\|\nabla\eta_0\|_{L^2(U)}^2
 \to c^2\beta\|\nabla\eta_0\|_{L^2(U)}^2.
\]
For the constant vector test $\ve_i$, the local variation correction vanishes and the exterior cutoff equals $1$ on $\operatorname{supp}\eta_0$. Consequently,
\[
 \int_\Oe2(b_\eps-1)|D\Ceps\ve_i|^2\,\mathrm dx
 \to2c\beta\int_U\eta_0|D\vh^i|^2\,\mathrm dx>0.
\]
Thus an $L^2$-small coefficient perturbation can retain a nonzero capacity-layer contribution. \Cref{prop:strongphase} supplies the $L^2_tH^1_x$ topology used in \cref{prop:weighted} to control this contribution.
\end{remark}

For the NSCH sequence, \cref{prop:strongphase} gives precisely \eqref{eq:coeffhyp} with $b_\eps=\nu(\Eeps\phi_\eps)$ and $b=\nu(\phi)$. The abstract pairing and lower bound therefore apply to the microscopic viscous stress.

\section{Velocity compactness and the coupled limit}\label{sec:limit}
With the phase compactness and the weighted capacity limit established, it remains to obtain strong compactness of the velocity and to pass to the coupled equations. These steps, together with the dissipation lower bound and the pressure reconstruction, complete the proof of \cref{thm:main}.

\subsection{Compactness through corrected modes}
Time compactness of the ambient velocity is obtained from corrected fixed-space modes, which yield strong $L^2$ compactness.

\begin{lemma}[Strong velocity compactness]\label{lem:velocitycompactness}
Under the assumptions of \cref{thm:main}, the sequence $\widetilde \vu_\eps$ is relatively compact in $L^2(0,T;\cH)$. Along a convergent subsequence it also satisfies the uniform weak-time convergence in \eqref{eq:strongtimetraces}, with a limit in $C_w([0,T];\cH)$ and initial value $\vu^0$.
\end{lemma}
\begin{proof}
Let $\{\vpsi_j\}_{j\geq1}\subset\Ssig$ be an $L^2$-orthonormal basis of $\cH$ obtained by orthonormalizing a countable dense family of smooth solenoidal fields. In the toroidal case, the constant vector fields are included in the initial dense family. For fixed $j$ define
\[
 c_{\eps,j}(t)=\int_\Oe \vu_\eps(t)\cdot\Ceps\vpsi_j\,\mathrm dx.
\]
Testing the momentum equation with $\theta(t)\Ceps\vpsi_j$, $\theta\in C_c^\infty(0,T)$, shows that $c_{\eps,j}'$ is the sum of the viscous, convective, forcing and capillary pairings. The estimates below are uniform in $\eps$ for each fixed $j$:
\[
\begin{aligned}
 \|2\nu(\phi_\eps)D\vu_\eps:D\Ceps\vpsi_j\|_{L^1(\Oe)}
 &\leq C_j\|D\vu_\eps\|_{L^2(\Oe)},\\
 \left|\int_\Oe(\vu_\eps\otimes\vu_\eps):\nabla\Ceps\vpsi_j\,\mathrm dx\right|
 &\leq C_j\|\vu_\eps\|_{L^2(\Oe)}^{1/2}
              \|\widetilde{\vu}_\eps\|_{H^1(\Omega)}^{3/2},\\
 \left|\int_\Oe\phi_\eps\nabla\mu_\eps\cdot\Ceps\vpsi_j\,\mathrm dx\right|
 &\leq C_j\|\phi_\eps\|_{L^6(\Oe)}\|\nabla\mu_\eps\|_{L^2(\Oe)},\\
 \left|\int_\Oe\vf_\eps\cdot\Ceps\vpsi_j\,\mathrm dx\right|
 &\leq C_j\|\vf_\eps\|_{L^2(\Oe)}.
\end{aligned}
\]
The first, third and fourth right-hand sides belong to $L^2(0,T)$. The convective bound belongs to $L^{4/3}(0,T)$ because $\widetilde{\vu}_\eps\in L^\infty_tL^2_x\cap L^2_tL^6_x$ and the Gagliardo--Nirenberg interpolation estimate gives $\|\vu_\eps\|_4^2\leq C\|\vu_\eps\|_2^{1/2}\|\widetilde{\vu}_\eps\|_{H^1}^{3/2}$. Therefore
\[
 \sup_\eps\|c_{\eps,j}\|_{W^{1,4/3}(0,T)}\leq C_j.
\]
In particular,
\[
 |c_{\eps,j}(t)-c_{\eps,j}(s)|
 \leq C_j|t-s|^{1/4},\quad s,t\in[0,T],
\]
and the family is relatively compact in $C([0,T])$ by the Arzel\`a--Ascoli theorem.

The corrected coefficient differs uniformly in time from the ordinary Fourier coefficient:
\[
 \sup_{t\in[0,T]}
 \left|\big(\widetilde{\vu}_\eps(t),\vpsi_j\big)_\Omega-c_{\eps,j}(t)\right|
 \leq \|\widetilde{\vu}_\eps\|_{L^\infty_tL^2_x}
       \|\vpsi_j-\Ceps\vpsi_j\|_{L^2(\Omega)}\to0.
\]
After a diagonal extraction, every fixed coefficient
$t\mapsto(\widetilde{\vu}_\eps(t),\vpsi_j)_\Omega$ therefore converges uniformly on $[0,T]$.

To pass from finitely many modes to the full velocity, let
\[
 \mathcal P_N^{\vu}\vz=\sum_{j=1}^N(\vz,\vpsi_j)_\Omega\vpsi_j,
 \quad \vz\in\cH.
\]
Since $\cV\hookrightarrow\hookrightarrow\cH$, the spectral tail numbers
\[
 \delta_N:=\sup_{\substack{\vz\in\cV\\ \|\vz\|_{H^1(\Omega)}\leq1}}
          \|(\operatorname{Id}-\mathcal P_N^{\vu})\vz\|_{L^2(\Omega)}
\]
satisfy $\delta_N\to0$. Hence
\[
 \int_0^T\|(\operatorname{Id}-\mathcal P_N^{\vu})
        \widetilde{\vu}_\eps(t)\|_2^2\,\mathrm dt
 \leq \delta_N^2\|\widetilde{\vu}_\eps\|_{L^2_tH^1_x}^2
 \leq C_T\delta_N^2.
\]
For fixed $N$, uniform convergence of the first $N$ coefficients yields compactness of
$\mathcal P_N^{\vu}\widetilde{\vu}_\eps$ in $L^2(0,T;\cH)$. Taking first $\eps\to0$ and then $N\to\infty$ proves relative compactness of $\widetilde{\vu}_\eps$ in $L^2(0,T;\cH)$. 

The same mode estimates also yield the weak time trace. For any finite linear combination $\vz_N=\sum_{j=1}^N\alpha_j\vpsi_j$,
\[
 \sup_{t\in[0,T]}
 \left|\big(\widetilde{\vu}_\eps(t)-\vu(t),\vz_N\big)_\Omega\right|\to0.
\]
Given $\vz\in\cH$, choose $\vz_N$ in the span of $\{\vpsi_1,\ldots,\vpsi_N\}$ with $\|\vz-\vz_N\|_{\cH}\to0$. Then
\[
 \sup_{t\in[0,T]}|((\widetilde{\vu}_\eps-\vu)(t),\vz-\vz_N)_\Omega|
 \leq C_T\|\vz-\vz_N\|_{\cH},
\]
uniformly in $\eps$. First letting $\eps\to0$ at fixed $N$ and then $N\to\infty$ yields the assertion for $\vz$. The limiting pairings are continuous in time and define a representative $\vu\in C_w([0,T];\cH)$. Finally, the convergence of the initial data and the corrected-test estimate identify $\vu(0)=\vu^0$.
\end{proof}

Together with the phase estimates, the corrected-mode compactness gives the convergences required for the weak limit passage.

\subsection{Homogenization process}
The compactness results above are sufficient to identify every term in the effective equations of \cref{thm:main}. The uniform estimates, \cref{lem:phaseL2}, \cref{prop:strongphase} and \cref{lem:velocitycompactness} give \eqref{eq:mainconvergences}--\eqref{eq:strongtimetraces}. Incompressibility and the outer no-slip condition pass to the weak limit in $L^2(0,T;\cV)$.

Let $\vpsi$ be a smooth solenoidal test with $\vpsi(T)=0$, and test the microscopic momentum equation by $\Ceps\vpsi$. Time differentiation commutes with the correction. The velocity, force and initial terms therefore converge by the strong $L^2$ corrector bound and \eqref{eq:data}. For the viscous term, apply \cref{prop:weighted} with
\[
 b_\eps=\nu(\Eeps\phi_\eps),\quad b=\nu(\phi),
 \quad \vw_\eps=\vu_\eps.
\]
This gives
\[
 \int_0^T\!\int_\Oe2\nu(\phi_\eps)D\vu_\eps:D\Ceps\vpsi\dx\dt
 \to\int_{\Dt}
        \big(2\nu(\phi)D\vu:D\vpsi+\nu(\phi)\B \vu\cdot\vpsi\big)\dx\dt.
\]

The convective corrector is handled by integration by parts, since the corrector gradients carry critical energy. Put $\vr_\eps=\Ceps\vpsi-\vpsi$. Using incompressibility and the zero trace of $\vu_\eps$, integration by parts yields
\[
 \int_0^T\!\int_\Oe(\vu_\eps\otimes \vu_\eps):\nabla\vr_\eps\dx\dt
 =-\int_0^T\!\int_\Oe (\vu_\eps\cdot\nabla)\vu_\eps\cdot\vr_\eps\dx\dt.
\]
By \cref{lem:corrector}, the right-hand side is bounded by
\[
 \|\vu_\eps\|_{L^2_tL^6(\Oe)}
 \|\nabla \vu_\eps\|_{L^2((0,T)\times\Oe)}
 \|\vr_\eps\|_{L^\infty_tL^3_x}\to0.
\]
This calculation is legitimate by $H^1$ approximation at fixed $\eps$ and all products belong to the displayed dual integrability classes. The remaining term with $\nabla\vpsi$ converges because $\widetilde \vu_\eps\otimes\widetilde \vu_\eps\to \vu\otimes \vu$ in $L^1(\Dt)$.

For the capillary force, the corrector error is bounded by
\[
 \left|\int_0^T\!\int_\Oe\phi_\eps\nabla\mu_\eps\cdot\vr_\eps\dx\dt\right|
 \leq\|\phi_\eps\|_{L^\infty_tL^6(\Oe)}
       \|\nabla\mu_\eps\|_{L^2((0,T)\times\Oe)}
       \|\vr_\eps\|_{L^2_tL^3_x}\to0.
\]
Against the fixed test $\vpsi$, $\Eeps\phi_\eps\vpsi$ converges strongly in $L^2$, while $\chi_\eps\nabla\mu_\eps\weak\nabla\mu$ in $L^2$. Thus the capillary term converges to $\lambda\phi\nabla\mu$. Combining these limits yields the first equation of \eqref{eq:limit} in solenoidal weak form.

For the phase equation, the restriction of a smooth ambient scalar test to $\Oe$ is admissible. 
Since $\Eeps\phi_\eps\to\phi$ strongly in $L^2_tH^1_x$ and $\widetilde\vu_\eps\to\vu$ strongly in $L^2_tL^2_x$, we have
\[
 (\Eeps\phi_\eps)\widetilde \vu_\eps\to\phi \vu
                    \quad\text{in }L^1(\Dt).
\]
Along the subsequence fixed in \eqref{eq:mainconvergences}, $\Eeps\phi_\eps\to\phi$ almost everywhere in $\mathcal D_T$. Since $m$ is bounded and continuous, $m(\Eeps\phi_\eps)\to m(\phi)$ almost everywhere. Bounded convergence against each fixed $L^2$ test, combined with weak convergence of $\chi_\eps\nabla\mu_\eps$, identifies the mobility flux:
\[
 \chi_\eps m(\phi_\eps)\nabla\mu_\eps
       \weak m(\phi)\nabla\mu\quad\text{in }L^2(\Dt).
\]
The chemical-potential identity \eqref{eq:limitchemical} is inherited by the limit, which completes the identification of the effective system. 
 To extend the limiting equations from smooth tests to their stated energy-space test classes, observe that
\begin{equation}\label{eq:limitfluxclasses}
 \vu\otimes \vu\in L^{4/3}_tL^2_x,\quad
 \phi\nabla\mu\in L^2_tL^{3/2}_x,\quad
 \phi \vu,\ m(\phi)\nabla\mu\in L^2_tL^2_x.
\end{equation}
The corresponding microscopic quantities satisfy the same uniform bounds. Distributional identification and weak compactness in these reflexive spaces identify their weak limits. For the momentum equation, $\vu\otimes\vu\in L^{4/3}_tL^2_x$ pairs with $\nabla\vpsi\in L^4_tL^2_x$, while $\phi\nabla\mu\in L^2_tL^{3/2}_x$ pairs with $\vpsi\in L^2_tL^3_x$, which follows from $\cV\hookrightarrow L^6(\Omega;\R^3)$. The viscous and Brinkman terms are continuous on $L^2(0,T;\cV)$. For the phase equation, $\phi\vu$ and $m(\phi)\nabla\mu$ belong to $L^2_tL^2_x$ and therefore pair with gradients of $H^1$ tests.

To make the density step explicit, let $\{\vpsi_j\}_{j\ge1}\subset\Ssig$ be dense in $\cV$ and let $\vpsi\in C^1([0,T];\cV)$ with $\vpsi(T)=0$. Finite-rank approximations of the form
\[
 \vpsi_m(t)=\sum_{j=1}^{N_m}\alpha_{m,j}(t)\vpsi_j,
 \qquad \alpha_{m,j}\in C^1([0,T]),\quad \alpha_{m,j}(T)=0,
\]
can be chosen so that $\vpsi_m\to\vpsi$ in $C^1([0,T];\cV)$. In particular, $\vpsi_m(0)\to\vpsi(0)$ in $\cH$. The flux classes above make every space--time term continuous under this convergence, while
\[
 |(\vu^0,\vpsi_m(0)-\vpsi(0))_\Omega|
 \le\|\vu^0\|_2\|\vpsi_m(0)-\vpsi(0)\|_2.
\]
The same finite-rank-in-space, smooth-in-time construction applies to the scalar test classes. Passing first with smooth finite-rank tests and then letting $m\to\infty$ yields the full weak formulation. The scalar boundary conditions therefore have the variational meaning specified in \cref{def:effectiveweak}.

Testing the phase equation by the spatial constant $1$ shows that the distributional derivative of $t\mapsto\int_\Omega\phi(t)\dx$ vanishes. The initial mean is the limit of the microscopic initial means: by \eqref{eq:fillingerror}, \eqref{eq:data} and $|\He|\to0$,
\[
 \int_{\Oe}\phi_\eps^0\,\dx
 =\int_\Omega \chi_\eps\phi_\eps^0\,\dx
 \to \int_\Omega\phi^0\,\dx.
\]
Since $\phi\in C([0,T];L^2(\Omega))$ by \eqref{eq:strongtimetraces}, the limiting mean has a continuous representative. Hence
\[
 \int_\Omega\phi(t)\dx=\int_\Omega\phi^0\dx,
 \quad t\in[0,T].
\]

The flux bounds imply $\partial_t\phi\in L^2(0,T;(H^1(\Omega))')$ and $\partial_t\vu\in L^{4/3}(0,T;\cV')$. The time representatives and initial traces were identified in \cref{lem:phaseL2,lem:velocitycompactness}. In addition, $\phi\in C_w([0,T];H^1(\Omega))$ by its $L^\infty_tH^1_x$ bound and strong $L^2$ time continuity. These facts establish the asserted weak-solution regularity.

\subsection{Dissipation and pressure}
The weighted lower-bound argument identifies the resistance contribution to the viscous dissipation in addition to the bulk strain.

The strong space--time convergences allow a further subsequence such that
\[
 \widetilde \vu_\eps(t)\to \vu(t)\text{ in }L^2(\Omega),\qquad
 \Eeps\phi_\eps(t)\to\phi(t)\text{ in }H^1(\Omega)
\]
for almost every $t\in(0,T)$. Intersect this set with the full-measure sets on which the energy inequalities for the countably many members of the selected sequence hold. On the resulting common set, both the kinetic and phase energies converge. For the phase energy, the integrals over the holes vanish by strong $H^1$ convergence and the fixed-domain embedding $H^1\hookrightarrow L^4$. The initial energy convergence follows in the same way as in \cref{subsec:main}.

Apply \eqref{eq:weightedliminf} on $(0,t)$ to obtain
\begin{equation}\label{eq:viscousdisslimit}
 \liminf_{\eps\to0}\int_0^t\!\int_\Oe2\nu(\phi_\eps)|D\vu_\eps|^2\dx\dd s
 \geq\int_0^t\!\int_\Omega
      \big(2\nu(\phi)|D\vu|^2+\nu(\phi)\B \vu\cdot \vu\big)\dx\dd s.
\end{equation}
Likewise,
\[
 \chi_\eps\sqrt{m(\phi_\eps)}\nabla\mu_\eps
       \weak\sqrt{m(\phi)}\nabla\mu\quad\text{in }L^2(\Dt),
\]
so weak lower semicontinuity gives the chemical dissipation. The force work converges in $L^1(\Dt)$ and hence its time primitives converge uniformly. The weighted lower bound holds on every fixed interval $(0,t)$ along the same subsequence, since the convergences used in \cref{prop:weighted} restrict to that interval. Taking the lower limit in the microscopic inequality on the common full-measure set proves \eqref{eq:limitenergy}.

To reconstruct the pressure, integrate the limiting solenoidal momentum identity in time. Define $\mathcal G(t)\in\W'$ by
\[
 \mathcal G(t)=\vu^0-\vu(t)+\int_0^t
 \Big[\vf-\lambda\phi\nabla\mu-\Div(\vu\otimes\vu)
      +\Div(2\nu(\phi)D\vu)-\nu(\phi)\B\vu\Big](s)\,\mathrm ds.
\]
The bounds in \eqref{eq:limitfluxclasses} and the energy estimate give
\[
 \|\mathcal G\|_{L^\infty(0,T;\W')}\leq C_T.
\]
Choose a countable set $\{\vpsi_j\}_{j\geq1}\subset\cV$ dense in $\cV$ with respect to the $H^1$ norm. For each $j$, the integrated weak momentum identity gives $\langle\mathcal G(t),\vpsi_j\rangle=0$ outside a measure-zero set $N_j$. On $[0,T]\setminus\bigcup_jN_j$ the identity holds for every $j$ and the continuity of $\mathcal G(t)\in\W'$ with respect to the $H^1$ norm then extends it to every $\vpsi\in\cV$. Thus one common full-measure set works for all solenoidal tests.

The fixed-domain divergence right inverse converts the annihilation of $\mathcal G(t)$ on $\cV$ into a gradient representation. If $\Omega$ is bounded, the Bogovski\u\i\ theorem provides a bounded linear map
\[
 \mathfrak B_\Omega:L^2_0(\Omega)\to H^1_0(\Omega;\mathbb R^3),
 \quad
 \Div(\mathfrak B_\Omega q)=q,
 \quad
 \|\mathfrak B_\Omega q\|_{H^1(\Omega)}\leq C_\Omega\|q\|_2,
\]
see \cite[Chapter~III, Section~3]{Galdi}. This Lipschitz-domain theorem applies in particular to the present $C^3$ boundary. On the torus, for $q\in L^2_0(\T)$ write its Fourier series $q=\sum_{k\in\mathbb Z^3\setminus\{0\}}\widehat q_k e^{2\pi i k\cdot x}$ and define the zero-mean solution of $\Delta\zeta_q=q$ by $\widehat\zeta_{q,k}=-(4\pi^2|k|^2)^{-1}\widehat q_k$. Then $\zeta_q\in H^2(\T)$ and
\[
 \|\nabla\zeta_q\|_{H^1(\T)}\leq C\|q\|_{L^2(\T)}.
\]
Setting $\mathfrak B_\Omega q=\nabla\zeta_q$ gives the same right-inverse estimate. For almost every $t$, define a linear functional on $L^2_0(\Omega)$ by
\[
 \Lambda_t(q)=-\langle\mathcal G(t),\mathfrak B_\Omega q\rangle.
\]
Then
$|\Lambda_t(q)|\leq C_\Omega\|\mathcal G(t)\|_{\W'}\|q\|_2$.
The Riesz representation theorem therefore yields a unique \(P(t)\in L^2_0(\Omega)\) such that
\[
 (P(t),q)_\Omega=\Lambda_t(q),
 \qquad
 \|P(t)\|_2\leq C_\Omega\|\mathcal G(t)\|_{\W'}.
\]
The map from \(\mathcal G(t)\) to \(P(t)\) is induced by fixed bounded linear operators. Hence $P$ is weakly measurable. Since $L^2_0(\Omega)$ is separable, the Pettis measurability theorem yields strong measurability. Consequently
\(P\in L^\infty(0,T;L^2_0(\Omega))\). For $\vxi\in\W$, the function
\[
 \vxi_0=\vxi-\mathfrak B_\Omega(\Div\vxi)
\]
belongs to $\cV$; in the toroidal case the constant solenoidal modes are already included in $\cV$. Hence
\[
 \langle\mathcal G(t),\vxi\rangle
 =-\int_\Omega P(t)\Div\vxi\,\mathrm dx
 =\langle\nabla P(t),\vxi\rangle.
\]
Consequently $\mathcal G=\nabla P$ in $L^\infty(0,T;\W')$. 
We understand $p=\partial_tP$ as a space--time distribution, defined by
\[
 \langle p,\zeta\rangle=-\int_0^T(P(t),\partial_t\zeta(t))_\Omega\,\mathrm dt,
 \quad \zeta\in C_c^\infty((0,T)\times\Omega).
\]
No Bochner time derivative of $P$ in $L^2_0(\Omega)$ is used or claimed. Differentiating $\mathcal G=\nabla P$ in time in the distributional sense and using the definition of $\mathcal G$ recovers the full momentum equation in \eqref{eq:limit}; this $p$ is the pressure distribution occurring there. This completes the proof of \cref{thm:main}.

\begin{proof}[Proof of \cref{cor:spheres}]
For $K=\overline{B(0,r)}$, put $s=|x|$ and $\vomega=x/s$. The decaying exterior field with boundary value $\ve_i$ is
\[
 \vh^i(x)=\left[\frac{3r}{4s}(\Ithree+\vomega\otimes\vomega)
          +\frac{r^3}{4s^3}(\Ithree-3\vomega\otimes\vomega)\right]\ve_i,
 \quad q^i(x)=\frac{3r}{2s^2}(\ve_i\cdot\vomega).
\]
A direct substitution gives $-\Delta\vh^i+\nabla q^i=0$ and $\Div\vh^i=0$ for $s>r$ and the formula gives $\vh^i=\ve_i$ at $s=r$ and $\vh^i(x)\to0$ as $s\to\infty$. Thus it satisfies \eqref{eq:exterior}. At $s=r$, the outward fluid normal is $-\vomega$ and the traction is $(2D\vh^i-q^i\Ithree)(-\vomega)=3\ve_i/(2r)$. Integrating over the sphere yields $\M \ve_i=6\pi r \ve_i$. Multiplication by $\beta\nu(\phi)$ gives the asserted resistance. Replacing $\nu$ by $\nu_{\mathrm{ref}}/2$ gives the final formula.
\end{proof}

The macroscopic viscous term retains the full symmetric stress, with $\nu(\phi)$ inside the divergence. Its coexistence with the positive resistance term distinguishes \eqref{eq:limit} from the Darcy system. The nonvanishing capacity energy in \eqref{eq:viscousdisslimit} is carried by microscopic boundary layers. Correspondingly, the velocity convergence in $L^2_tH^1_x$ is weak, while the phase convergence in that space is strong.

\section{The vanishing-capillarity limit}\label{sec:vanishing}
We prove \cref{thm:vanishing} using the same geometric and scalar arguments, while keeping track of the normalization in the equations. For each fixed $\eps$, the solenoidal momentum identity determines a pressure distribution $p_\eps$, unique up to an additive function of time, by the standard de Rham/Bogovski\u\i{} argument on $\Oe$. No estimate uniform in $\eps$ is needed here.
Define the rescaled microscopic pressure by $\pi_\eps=\lambda_\eps^{-1/2}p_\eps$ in distributions. Substituting $\vu_\eps=\sqrt{\lambda_\eps}\,\vv_\eps$ and $\vf_\eps=\sqrt{\lambda_\eps}\,\vg_\eps$ into the momentum equation and dividing it by $\sqrt{\lambda_\eps}$ leaves the acceleration and viscous terms at order one, while convection and capillarity acquire the factor $\sqrt{\lambda_\eps}$. In the phase equation only the transport velocity is rescaled, producing the same factor in front of advection. Thus
\begin{equation}\label{eq:scaledmicro}
 \left\{\begin{aligned}
 \partial_t\vv_\eps+\sqrt{\lambda_\eps}\Div(\vv_\eps\otimes \vv_\eps)
 -\Div(2\nu(\phi_\eps)D\vv_\eps)+\nabla\pi_\eps
    &=\vg_\eps-\sqrt{\lambda_\eps}\phi_\eps\nabla\mu_\eps,\\
 \Div \vv_\eps&=0,\\
 \partial_t\phi_\eps+\sqrt{\lambda_\eps}\Div(\phi_\eps \vv_\eps)
    &=\Div(m(\phi_\eps)\nabla\mu_\eps),\\
 \mu_\eps&=-\Delta\phi_\eps+F'(\phi_\eps).
 \end{aligned}\right.
\end{equation}
Dividing the microscopic energy inequality by $\lambda_\eps$ and using $\vu_\eps=\sqrt{\lambda_\eps}\,\vv_\eps$ and $\vf_\eps=\sqrt{\lambda_\eps}\,\vg_\eps$ gives
\begin{equation}\label{eq:scaledenergy}
 \begin{aligned}
 \tfrac12\|\vv_\eps(t)\|_{L^2(\Oe)}^2+\calF_\Oe(\phi_\eps(t))
 +\int_0^t\!\int_\Oe
       \big(2\nu(\phi_\eps)|D\vv_\eps|^2
           +m(\phi_\eps)|\nabla\mu_\eps|^2\big)\dx\dd s\\
 \leq\tfrac12\|\vv_\eps^0\|_{L^2(\Oe)}^2+\calF_\Oe(\phi_\eps^0)
       +\int_0^t\!\int_\Oe \vg_\eps\cdot \vv_\eps\dx\dd s.
 \end{aligned}
\end{equation}
Hence \eqref{eq:uniformenergy}--\eqref{eq:muFbounds} hold with $\vv_\eps$ in place of $\vu_\eps$. The phase time-derivative bound remains valid, with its advective flux multiplied by $\sqrt{\lambda_\eps}$. The chemical-potential identity is unchanged, so \cref{lem:phaseL2} and \cref{prop:strongphase} give the same phase and viscosity convergences. The corrected-mode argument in \cref{lem:velocitycompactness} applies to $\vv_\eps$, since the additional factors $\sqrt{\lambda_\eps}$ are bounded. Thus the convergences stated in \eqref{eq:mainconvergences}--\eqref{eq:strongtimetraces} hold with $\vv_\eps,\vv$ in place of $\vu_\eps,\vu$. The remaining limit passage concerns the three terms carrying the factor $\sqrt{\lambda_\eps}$: momentum convection, capillary forcing and phase advection.

Testing the first equation of \eqref{eq:scaledmicro} by $\Ceps\vpsi$, the convective contribution tends to zero because
\[
 \sqrt{\lambda_\eps}
 \left|\int_0^T\!\int_\Oe(\vv_\eps\otimes \vv_\eps):\nabla\Ceps\vpsi\dx\dt\right|
 \leq C_{\vpsi} T^{1/4}\sqrt{\lambda_\eps}
          \|\vv_\eps\|_{L^{8/3}(0,T;L^4(\Oe))}^2\to0.
\]
The norm is uniformly bounded by the energy estimate and interpolation. The capillary contribution is bounded by
\[
 C_{\vpsi} T^{1/2}\sqrt{\lambda_\eps}
 \|\phi_\eps\|_{L^\infty_tL^6(\Oe)}
 \|\nabla\mu_\eps\|_{L^2((0,T)\times\Oe)}\to0.
\]
The viscous term converges to the bulk strain plus $\nu(\phi)\B \vv$ by \cref{prop:weighted}. The acceleration has coefficient one and survives. Passing to the other terms proves the unsteady Stokes--Brinkman equation.

The phase advection vanishes in distributions because
\[
 \sqrt{\lambda_\eps}\,
 \|(\Eeps\phi_\eps)\widetilde \vv_\eps\|_{L^2_tL^2_x}\to0.
\]
The mobility and chemical-potential terms pass as before, giving the unadvected Cahn--Hilliard equation in \eqref{eq:vanishinglimit}. The energy inequality follows from \eqref{eq:scaledenergy}, the weighted viscous lower bound, and the common full-measure endpoint argument used for \eqref{eq:limitenergy}. To identify the time derivative, fix $\theta\in C_c^\infty(0,T)$ and $\vz\in\cV$. The test field $\theta(t)\vz$ is admissible in the limiting momentum identity. Define $L(t)\in\cV'$ for almost every $t$ by
\[
 \langle L(t),\vz\rangle
 =\int_\Omega \vg(t)\cdot\vz\,\mathrm dx
 -\int_\Omega\bigl(2\nu(\phi)D\vv:D\vz
       +\nu(\phi)\B\vv\cdot\vz\bigr)\,\mathrm dx.
\]
The coefficient bounds and the energy estimate give
\[
 \|L(t)\|_{\cV'}
 \le C\bigl(\|\vg(t)\|_2+\|\vv(t)\|_{H^1}\bigr),
 \qquad L\in L^2(0,T;\cV').
\]
Hence
\[
 -\int_0^T(\vv,\vz)_\Omega\theta'\,\mathrm dt
 =\int_0^T\langle L,\vz\rangle\theta\,\mathrm dt,
\]
so $\partial_t\vv=L$ in $\mathcal D'(0,T;\cV')$. Since $\vv\in L^2(0,T;\cV)$, the Hilbert-space energy lemma yields $\vv\in C([0,T];\cH)$. The same fixed-domain pressure reconstruction used above gives $\pi=\partial_t\Pi$ with $\Pi\in L^\infty(0,T;L^2_0(\Omega))$.

Finally, $\vu_\eps=\sqrt{\lambda_\eps}\vv_\eps$ and the uniform normalized energy bounds give
\[
 \|\widetilde \vu_\eps\|_{L^\infty_tL^2_x}
  +\|\widetilde \vu_\eps\|_{L^2_tH^1_x}\leq C_T\sqrt{\lambda_\eps}\to0.
\]
This completes the proof of \cref{thm:vanishing}. The phase still affects the limiting velocity through $\nu(\phi)$, but the limiting velocity no longer transports the phase. This one-way coupling follows from the prescribed normalization together with the critical homogenization limit.

\Cref{app:existence} supplies the fixed-$\eps$ existence result used in \cref{thm:main}.

\appendix
\section{Weak solutions on a fixed perforated domain}\label[appendix]{app:existence}
The microscopic solution class is nonempty under the constitutive assumptions of \cref{subsec:model}. Constants in this appendix may depend on the fixed parameter $\eps$.

\begin{proposition}[Fixed-parameter existence]\label{prop:existence}
Let $\eps>0$ be fixed so that the holes are disjoint, let $\lambda_\eps>0$, and assume \eqref{eq:constitutive}. For any $\vu_\eps^0\in \cH_\eps$, $\phi_\eps^0\in H^1(\Oe)$ and $\vf_\eps\in L^2(0,T;L^2(\Oe;\R^3))$, there exists a finite-energy weak solution in the sense of \cref{def:weak} on $[0,T]$.
\end{proposition}
\begin{proof}
Set $G=\Oe$ and keep $\eps>0$ fixed. Let $\{\vw_j\}_{j\geq1}$ be an $L^2(G;\mathbb R^3)$-orthonormal Stokes eigenbasis of $\cH_\eps$ with $\vw_j\in\cV_\eps$, and let $\{\zeta_j\}_{j\geq1}$ be an $L^2(G)$-orthonormal Neumann Laplacian eigenbasis, with the normalized constant as the first mode. In the periodic case the eigenproblems are posed on the perforated torus with periodic identification and zero velocity trace on the holes. Define
\[
 X_n=\operatorname{span}\{\vw_1,\ldots,\vw_n\},
 \quad
 Y_n=\operatorname{span}\{\zeta_1,\ldots,\zeta_n\},
\]
and denote the corresponding orthogonal projections by
$\mathcal Q_n^G$ and $\mathcal P_n^G$. Spectral theory gives
\[
 \mathcal Q_n^G \vz\to \vz\quad\text{in }\cV_\eps,
 \quad
 \mathcal P_n^G\eta\to\eta\quad\text{in }H^1(G)
\]
for $\vz\in\cV_\eps$ and $\eta\in H^1(G)$. The spectral projections are uniformly bounded in the corresponding energy spaces: for all $n$,
\[
 \|\mathcal P_n^G\eta\|_{H^1(G)}\le C\|\eta\|_{H^1(G)},
 \qquad
 \|\mathcal Q_n^G\vz\|_{H^1(G)}\le C_\eps\|\vz\|_{H^1(G)}.
\]

We seek
\[
 \vu_n(t)=\sum_{j=1}^na_j(t)\vw_j,
 \quad
 \phi_n(t)=\sum_{j=1}^nb_j(t)\zeta_j,
 \quad
 \mu_n(t)=\sum_{j=1}^nd_j(t)\zeta_j,
\]
such that, for any $\vw\in X_n$ and $\zeta,\eta\in Y_n$,
\[
\left\{\begin{aligned}
 (\partial_t\vu_n,\vw)_G
 &-((\vu_n\otimes\vu_n),\nabla\vw)_G
 +(2\nu(\phi_n)D\vu_n,D\vw)_G\\
 &= (\vf_\eps,\vw)_G-\lambda_\eps(\phi_n\nabla\mu_n,\vw)_G,\\
 (\partial_t\phi_n,\zeta)_G
 &-(\phi_n\vu_n,\nabla\zeta)_G
 +(m(\phi_n)\nabla\mu_n,\nabla\zeta)_G=0,\\
 (\mu_n,\eta)_G
 &=(\nabla\phi_n,\nabla\eta)_G+(F'(\phi_n),\eta)_G.
\end{aligned}\right.
\]
The algebraic chemical-potential equation determines $d(t)$ continuously from $b(t)$. Substitution into the first two equations produces a finite-dimensional Carath\'eodory system for $(a,b)$ with measurable forcing and locally Lipschitz state dependence. Hence a local absolutely continuous solution exists. We take
\[
 \vu_n(0)=\mathcal Q_n^G\vu_\eps^0,
 \quad
 \phi_n(0)=\mathcal P_n^G\phi_\eps^0.
\]
Since \(\mathcal P_n^G\phi_\eps^0\to\phi_\eps^0\) in \(H^1(G)\), and hence in \(L^4(G)\), the quartic phase energies of the projected initial data converge to \(\calF_G(\phi_\eps^0)\).

Testing the three Galerkin equations by $\vu_n$, $\lambda_\eps\mu_n$ and $\lambda_\eps\partial_t\phi_n$, respectively, gives the exact finite-dimensional identity
\[
 \frac{\mathrm d}{\mathrm dt}\mathcal E_{G,\lambda_\eps}(\vu_n,\phi_n)
 +\int_G\Big(2\nu(\phi_n)|D\vu_n|^2
 +\lambda_\eps m(\phi_n)|\nabla\mu_n|^2\Big)\,\mathrm dx
 =\int_G\vf_\eps\cdot\vu_n\,\mathrm dx.
\]
Here $((\vu_n\otimes\vu_n),\nabla\vu_n)_G=0$ and the two capillary transport terms cancel because
$(\phi_n\nabla\mu_n,\vu_n)_G=(\phi_n\vu_n,\nabla\mu_n)_G$.
Young's inequality and Gronwall's lemma yield
\[
 \|\vu_n\|_{L^\infty_tL^2(G)}
 +\|\vu_n\|_{L^2_tH^1(G)}
 +\|\phi_n\|_{L^\infty_tH^1(G)}
 +\|\nabla\mu_n\|_{L^2_tL^2(G)}\leq C_{\eps,T}.
\]
The uniform \(L^\infty(0,T;H^1(G))\) bound for \(\phi_n\), together with the fixed-domain embedding \(H^1(G)\hookrightarrow L^6(G)\) and the quartic identity \(F'(s)=s^3-s\), gives
\[
 \|F'(\phi_n)\|_{L^\infty(0,T;L^2(G))}\le C_{\eps,T}.
\]
Because the constant mode belongs to $Y_n$,
$(\mu_n)_G=(F'(\phi_n))_G$. Combining this mean-value bound with the dissipation estimate for \(\nabla\mu_n\) and the fixed-domain Poincar\'e inequality yields
\[
 \|\mu_n\|_{L^2(0,T;H^1(G))}\le C_{\eps,T}.
\]
These bounds extend the finite-dimensional solution to all of $[0,T]$.

The energy bounds also control the time derivatives needed for compactness. For $\vw\in\cV_\eps$, use $\mathcal Q_n^G\vw$ as a Galerkin test. The coefficient and projection bounds give
\[
\begin{aligned}
 |((\vu_n\otimes\vu_n),\nabla\mathcal Q_n^G\vw)_G|
 &\leq C_\eps\|\vu_n\|_{L^2}^{1/2}\|\vu_n\|_{H^1}^{3/2}\|\vw\|_{H^1},\\
 |(\phi_n\nabla\mu_n,\mathcal Q_n^G\vw)_G|
 &\leq C_\eps\|\phi_n\|_{H^1}\|\nabla\mu_n\|_2\|\vw\|_{H^1}.
\end{aligned}
\]
The viscous and force terms belong to $L^2(0,T;\cV_\eps')$. Therefore
\[
 \|\partial_t\vu_n\|_{L^{4/3}(0,T;\cV_\eps')}\leq C_{\eps,T}.
\]
Similarly, for $\zeta\in H^1(G)$,
\[
 |(\phi_n\vu_n,\nabla\mathcal P_n^G\zeta)_G|
 \leq C_\eps\|\phi_n\|_{H^1}\|\vu_n\|_{H^1}\|\zeta\|_{H^1},
\]
and the mobility term has the same $L^2$ time integrability. Hence
\[
 \|\partial_t\phi_n\|_{L^2(0,T;(H^1(G))')}\leq C_{\eps,T}.
\]

Apply Simon's compactness theorem \cite[Corollary~4, p.~85]{Simon} to
$\cV_\eps\hookrightarrow\hookrightarrow\cH_\eps\hookrightarrow\cV_\eps'$ and to
$H^1(G)\hookrightarrow\hookrightarrow L^2(G)\hookrightarrow(H^1(G))'$. After subsequence extraction,
\[
\begin{aligned}
 \vu_n&\to\vu_\eps &&\text{strongly in }L^2(0,T;L^2(G;\mathbb R^3)),\\
 \vu_n&\rightharpoonup\vu_\eps &&\text{weakly in }L^2(0,T;\cV_\eps),\\
 \phi_n&\to\phi_\eps &&\text{strongly in }L^2(0,T;L^2(G)),\\
 \phi_n&\stackrel{*}{\rightharpoonup}\phi_\eps &&\text{in }L^\infty(0,T;H^1(G)),\\
 \mu_n&\rightharpoonup\mu_\eps &&\text{weakly in }L^2(0,T;H^1(G)).
\end{aligned}
\]

The phase-gradient convergence is strengthened by the chemical-potential equation. Here $\phi_\eps$ denotes the fixed-$\eps$ Galerkin limit obtained above. Since
$\mathcal P_n^G\phi_\eps\to\phi_\eps$ in $L^2(0,T;H^1(G))$, the admissible test
$\phi_n-\mathcal P_n^G\phi_\eps\in Y_n$ yields
\[
\begin{aligned}
 \int_0^T\!\!\int_G|\nabla\phi_n|^2\,\mathrm dx\,\mathrm dt
 ={}&\int_0^T\!\!\int_G\nabla\phi_n\cdot\nabla\mathcal P_n^G\phi_\eps\,\mathrm dx\,\mathrm dt\\
 &+\int_0^T\!\!\int_G(\mu_n-F'(\phi_n))
       (\phi_n-\mathcal P_n^G\phi_\eps)\,\mathrm dx\,\mathrm dt.
\end{aligned}
\]
The second term tends to zero, while the first tends to
$\|\nabla\phi_\eps\|_{L^2((0,T)\times G)}^2$. Therefore
\[
 \phi_n\to\phi_\eps\quad\text{strongly in }L^2(0,T;H^1(G)).
\]
In particular, after extraction, $\phi_n\to\phi_\eps$ almost everywhere and strongly in $L^2_tL^6_x$. The polynomial structure of $F'$ and the uniform $L^\infty_tL^6_x$ bound give
$F'(\phi_n)\to F'(\phi_\eps)$ strongly in $L^2_tL^2_x$.

The coefficient-weighted weak limits follow from almost-everywhere phase convergence and the uniform coefficient bounds. For any fixed \(L^2\) test tensor or vector, dominated convergence gives strong \(L^2\) convergence after multiplication by \(\nu(\phi_n)-\nu(\phi_\eps)\) or \(m(\phi_n)-m(\phi_\eps)\). Hence
\[
 \nu(\phi_n)D\vu_n\rightharpoonup\nu(\phi_\eps)D\vu_\eps,
 \qquad
 m(\phi_n)\nabla\mu_n\rightharpoonup m(\phi_\eps)\nabla\mu_\eps
\]
weakly in \(L^2\). The strong $L^2$ convergence of $\vu_n$ and the uniform energy bounds identify the convective term. Moreover,
\[
 \|\phi_n\nabla\mu_n\|_{L^2(0,T;L^{3/2}(G))}
 \le \|\phi_n\|_{L^\infty(0,T;L^6(G))}
      \|\nabla\mu_n\|_{L^2(0,T;L^2(G))}
 \le C_{\varepsilon,T}.
\]
Since $\phi_n\to\phi_\varepsilon$ strongly in $L^2(0,T;L^6(G))$ and
$\nabla\mu_n\rightharpoonup\nabla\mu_\varepsilon$ weakly in $L^2$, the product converges distributionally to
$\phi_\varepsilon\nabla\mu_\varepsilon$. The uniform bound above and uniqueness of the distributional limit therefore give
\[
 \phi_n\nabla\mu_n\rightharpoonup
 \phi_\varepsilon\nabla\mu_\varepsilon
 \quad\text{weakly in }L^2(0,T;L^{3/2}(G)).
\]
Similarly, the convective and phase-transport products may be identified in their natural reflexive spaces. Passing to the limit in the Galerkin equations and then using density gives the weak identities in \cref{def:weak}.

Before taking the lower limit in the dissipation, we identify the weighted weak limits. Since
\(\phi_n\to\phi_\eps\) almost everywhere, the coefficients
\(\sqrt{\nu(\phi_n)}\) and \(\sqrt{m(\phi_n)}\) converge almost everywhere and remain uniformly bounded. Dominated convergence against fixed \(L^2\) tests, combined with
\(D\vu_n\rightharpoonup D\vu_\eps\) and
\(\nabla\mu_n\rightharpoonup\nabla\mu_\eps\), yields
\[
 \sqrt{\nu(\phi_n)}D\vu_n
 \rightharpoonup
 \sqrt{\nu(\phi_\eps)}D\vu_\eps,\qquad
 \sqrt{m(\phi_n)}\nabla\mu_n
 \rightharpoonup
 \sqrt{m(\phi_\eps)}\nabla\mu_\eps
\]
weakly in \(L^2\). To pass the finite-dimensional energy identity to the limit at the time endpoints, choose a further subsequence such that
\[
 \vu_n(t)\to\vu_\eps(t)\quad\text{in }L^2(G;\mathbb R^3),\qquad
 \phi_n(t)\to\phi_\eps(t)\quad\text{in }H^1(G)
\]
for almost every $t\in(0,T)$; this follows from the strong space--time convergences above. The $H^1(G)\hookrightarrow L^4(G)$ embedding then gives convergence of the quartic phase energy at those times. The projected initial data converge in the same norms and
\(
 \int_0^t(\vf_\eps,\vu_n)_G\,\mathrm ds
 \to\int_0^t(\vf_\eps,\vu_\eps)_G\,\mathrm ds
\)
for every $t$ because $\vu_n\to\vu_\eps$ strongly in $L^2_tL^2_x$. Integrating the Galerkin energy identity on $(0,t)$ and taking the lower limit in the two dissipation terms therefore yields the energy inequality for almost every $t$.

The derivative bounds imply
$\vu_\eps\in C_w([0,T];\cH_\eps)$ and
$\phi_\eps\in C([0,T];L^2(G))$. The latter follows from the standard Lions--Magenes argument applied to
$\phi_\eps\in L^2_tH^1_x$ and
$\partial_t\phi_\eps\in L^2_t(H^1)'$.
The convergence of the projected initial data identifies the prescribed initial values. Thus $(\vu_\eps,\phi_\eps,\mu_\eps)$ is a finite-energy weak solution on $[0,T]$.
\end{proof}

\paragraph{Data availability} 
No new data were created or analysed in this study.

\section*{Declarations}

\paragraph{Conflicts of Interest}
The authors declare no conflicts of interest.

\end{document}